\documentclass[12pt]{amsart}
\date{\today}

\input xymatrix
\xyoption{all}

\usepackage{amsmath,amsthm,amsfonts,amssymb,mathrsfs}

\usepackage{color}

\usepackage{amssymb, cite}
\usepackage[colorlinks,plainpages,citecolor=magenta, linkcolor=blue , backref]{hyperref}%,bookmarksnumbered,

\usepackage{hyperref}

\newtheorem{theorem}{Theorem}[section]
\newtheorem{lemma}[theorem]{Lemma}

\newtheorem{proposition}[theorem]{Proposition}
\newtheorem{corollary}[theorem]{Corollary}

\theoremstyle{definition}

\newtheorem{example}[theorem]{Example}

\theoremstyle{remark}
\newtheorem{remark}[theorem]{Remark}

\numberwithin{equation}{section}

\begin{document}

\title[On the monoid of cofinite partial isometries of $\mathbb{N}$ with a bounded finite noise]{On the monoid of cofinite partial isometries of $\mathbb{N}$ with a bounded finite noise}

\author[O.~Gutik and P.~Khylynskyi]{Oleg~Gutik and Pavlo~Khylynskyi}
\address{Faculty of Mathematics, Ivan Franko National
University of Lviv, Universytetska 1, Lviv, 79000, Ukraine}
\email{oleg.gutik@lnu.edu.ua, ovgutik@yahoo.com, pavlo.khylynskyi@lnu.edu.ua}



\begin{abstract}
In the paper we study algebraic properties of the monoid $\mathbf{I}\mathbb{N}_{\infty}^{\boldsymbol{g}[j]}$ of cofinite partial isometries of the set of positive integers $\mathbb{N}$ with the bounded finite noise $j$. For the monoids $\mathbf{I}\mathbb{N}_{\infty}^{\boldsymbol{g}[j]}$ we prove counterparts of some classical results of Eberhart and Selden describing the closure
of the bicyclic semigroup in a locally compact topological inverse semigroup.  In particular we show that for any positive integer $j$ every Hausdorff shift-continuous  topology $\tau$ on $\mathbf{I}\mathbb{N}_{\infty}^{\boldsymbol{g}[j]}$ is discrete and  if $\mathbf{I}\mathbb{N}_{\infty}^{\boldsymbol{g}[j]}$ is a proper dense subsemigroup of a Hausdorff semitopological semigroup $S$, then $S\setminus \mathbf{I}\mathbb{N}_{\infty}^{\boldsymbol{g}[j]}$  is a closed ideal of $S$, and moreover if $S$ is a topological inverse semigroup then $S\setminus \mathbf{I}\mathbb{N}_{\infty}^{\boldsymbol{g}[j]}$ is a topological group. Also we describe the algebraic and topological structure of the closure of the monoid $\mathbf{I}\mathbb{N}_{\infty}^{\boldsymbol{g}[j]}$ in a locally compact topological inverse semigroup.
\end{abstract}

\maketitle

%\tableofcontents

\section{Introduction and preliminaries}

In this paper we shall follow the terminology of \cite{Carruth-Hildebrant-Koch-1983-1986, Clifford-Preston-1961-1967, Lawson-1998, Ruppert-1984}.
We shall denote the first infinite cardinal by $\omega$ and the cardinality of a set $A$ by $|A|$. By $\operatorname{cl}_X(A)$ we denote the closure of subset $A$ in a topological space $X$.
 %For any positive integer $n$ by $\mathscr{S}_n$ we denote the group of permutations of the set $\{1,\ldots,n\}$.

%We shall say that a non-empty subset $A$ of  a semigroup $S$ \emph{generates} $S$, or $A$ is a \emph{set of generators} of $S$, if for any $s\in S$ there exist $a_1,\ldots,a_k\in A$ such that $s=a_1\cdots a_k$. For any non-empty subset $A$ of a semigroup $S$ by $\langle A\rangle$ we denote a subsemigroup of $S$ which is generated by $A$.

A semigroup $S$ is called {\it inverse} if for any
element $x\in S$ there exists a unique $x^{-1}\in S$ such that
$xx^{-1}x=x$ and $x^{-1}xx^{-1}=x^{-1}$. The element $x^{-1}$ is
called the {\it inverse of} $x\in S$. If $S$ is an inverse
semigroup, then the function $\operatorname{inv}\colon S\to S$
which assigns to every element $x$ of $S$ its inverse element
$x^{-1}$ is called the {\it inversion}.

If $S$ is a semigroup, then we shall denote the subset of all
idempotents in $S$ by $E(S)$. If $S$ is an inverse semigroup, then
$E(S)$ is closed under multiplication and we shall refer to $E(S)$ as a
\emph{band} (or the \emph{band of} $S$).

If $S$ is an inverse semigroup then the semigroup operation on $S$ determines the following partial order $\preccurlyeq$
on $S$: $s\preccurlyeq t$ if and only if there exists $e\in E(S)$ such that $s=te$. This order is
called the {\em natural partial order} on $S$ and it induces the {\em natural partial order} on the semilattice $E(S)$ \cite{Wagner-1952}.

An inverse subsemigroup $T$ of an inverse semigroup $S$ is called \emph{full} if $E(T)=E(S)$.

A congruence $\mathfrak{C}$ on a semigroup $S$ is called a \emph{group congruence} if the quotient semigroup $S/\mathfrak{C}$ is a group. Any inverse semigroup $S$ admits the \emph{minimum group congruence} $\mathfrak{C}_{\mathbf{mg}}$:
\begin{equation*}
  a\mathfrak{C}_{\mathbf{mg}}b \quad \hbox{if and only if} \quad \hbox{there exists} \quad e\in E(S) \quad \hbox{such that} \quad ea=eb.
\end{equation*}
Also, we say that a semigroup homomorphism $\mathfrak{h}\colon S\to T$ is a \emph{group homomorphism} if the image $(S)\mathfrak{h}$ is a group, and $\mathfrak{h}\colon S\to T$ is \emph{trivial} if it is either an isomorphism or annihilating.

The bicyclic monoid ${\mathscr{C}}(p,q)$ is the semigroup with the identity $1$ generated by two elements $p$ and $q$ subjected only to the condition $pq=1$. The semigroup operation on ${\mathscr{C}}(p,q)$ is determined as
follows:
\begin{equation*}
    q^kp^l\cdot q^mp^n=q^{k+m-\min\{l,m\}}p^{l+n-\min\{l,m\}}.
\end{equation*}
It is well known that the bicyclic monoid ${\mathscr{C}}(p,q)$ is a bisimple (and hence simple) combinatorial $E$-unitary inverse semigroup and every non-trivial congruence on ${\mathscr{C}}(p,q)$ is a group congruence \cite{Clifford-Preston-1961-1967}.

If $\alpha\colon X\rightharpoonup Y$ is a partial map, then we shall denote the domain and the range of $\alpha$ by $\operatorname{dom}\alpha$ and $\operatorname{ran}\alpha$, respectively. A partial map $\alpha\colon X\rightharpoonup Y$ is called \emph{cofinite} if both sets $X\setminus\operatorname{dom}\alpha$ and $Y\setminus\operatorname{ran}\alpha$ are finite.

Let $\mathscr{I}_\lambda$ denote the set of all partial one-to-one transformations of a non-zero  cardinal $\lambda$ together
with the following semigroup operation:
\begin{equation*}
x(\alpha\beta)=(x\alpha)\beta \quad \hbox{if} \quad x\in\operatorname{dom}(\alpha\beta)=\{
y\in\operatorname{dom}\alpha\colon y\alpha\in\operatorname{dom}\beta\}, \quad  \hbox{for} \;
\alpha,\beta\in\mathscr{I}_\lambda.
\end{equation*}
 The semigroup
$\mathscr{I}_\lambda$ is called the \emph{symmetric inverse} (\emph{monoid})
\emph{semigroup} over the cardinal $\lambda$~(see \cite{Clifford-Preston-1961-1967}). The symmetric inverse
semigroup was introduced by Wagner~\cite{Wagner-1952} and it plays
a major role in the theory of semigroups. By $\mathscr{I}^{\mathrm{cf}}_\lambda$ is denoted a
subsemigroup of injective partial selfmaps of $\lambda$ with
cofinite domains and ranges in $\mathscr{I}_\lambda$. Obviously, $\mathscr{I}^{\mathrm{cf}}_\lambda$ is an inverse
submonoid of the semigroup $\mathscr{I}_\lambda$. The
semigroup $\mathscr{I}^{\mathrm{cf}}_\lambda$  is called the \emph{monoid of
injective partial cofinite selfmaps} of $\lambda$ \cite{Gutik-Repovs-2015}.

%\smallskip

A partial transformation $\alpha\colon (X,d)\rightharpoonup (X,d)$ of a metric space $(X,d)$ is called \emph{isometric} or a \emph{partial isometry}, if $d(x\alpha,y\alpha)=d(x,y)$ for all $x,y\in \operatorname{dom}\alpha$. It is obvious that the composition of two partial isometries of a metric space $(X,d)$ is a partial isometry, and the converse partial map to a partial isometry is a partial isometry, too. Hence the set of partial isometries of a metric space $(X,d)$ with the operation of composition of partial isometries is an inverse submonoid of the symmetric inverse monoid over the cardinal $|X|$. Also, it is obvious that the set of partial cofinite isometries of a metric space $(X,d)$ with the operation the composition of partial isometries is an inverse submonoid of the monoid of injective partial cofinite selfmaps of the cardinal $|X|$.

%\smallskip

We endow the sets $\mathbb{N}$ and $\mathbb{Z}$ with the standard linear order.

The semigroup $\mathbf{ID}_{\infty}$ of all partial cofinite isometries of the set of integers $\mathbb{Z}$ with the usual metric $d(n,m)=|n-m|$, $n,m\in \mathbb{Z}$, was studied in the  papers \cite{Bezushchak-2004, Bezushchak-2008, Gutik-Savchuk-2017}.

Let $\mathbf{I}\mathbb{N}_{\infty}$ be the set of all partial cofinite isometries of the set of positive integers $\mathbb{N}$ with the usual metric $d(n,m)=|n-m|$, $n,m\in \mathbb{N}$. Then $\mathbf{I}\mathbb{N}_{\infty}$ with the operation of composition of partial isometries is an inverse submonoid of $\mathscr{I}_\omega$. The semigroup $\mathbf{I}\mathbb{N}_{\infty}$ of all partial cofinite isometries of positive integers is studied in \cite{Gutik-Savchuk-2018}. There we described the Green relations on the semigroup $\mathbf{I}\mathbb{N}_{\infty}$, its band and proved that $\mathbf{I}\mathbb{N}_{\infty}$ is a simple $E$-unitary $F$-inverse semigroup. Also in \cite{Gutik-Savchuk-2018}, the least group congruence $\mathfrak{C}_{\mathbf{mg}}$ on $\mathbf{I}\mathbb{N}_{\infty}$ is described and there it is proved that the quotient-semigroup  $\mathbf{I}\mathbb{N}_{\infty}/\mathfrak{C}_{\mathbf{mg}}$ is isomorphic to the additive group of integers $\mathbb{Z}(+)$. An example of a non-group congruence on the semigroup $\mathbf{I}\mathbb{N}_{\infty}$ is presented. Also it is  proved that a congruence on the semigroup $\mathbf{I}\mathbb{N}_{\infty}$ is a group congruence if and only if its restriction onto an isomorphic  copy of the bicyclic semigroup in $\mathbf{I}\mathbb{N}_{\infty}$ is a group congruence.
In \cite{Gutik-Savchuk-2020} it was shown that the monoid $\mathbf{I}\mathbb{N}_{\infty}$ does not embed isomorphically into the semigroup $\mathbf{ID}_{\infty}$.  Moreover every non-annihilating homomorphism $\mathfrak{h}\colon \mathbf{I}\mathbb{N}_{\infty}\to\mathbf{ID}_{\infty}$ has the following property:   the image $(\mathbf{I}\mathbb{N}_{\infty})\mathfrak{h}$ is isomorphic  either to $\mathbb{Z}_2$ or to $\mathbb{Z}(+)$. Also it is proved that  $\mathbf{I}\mathbb{N}_{\infty}$ does not have a finite set of generators, and moreover it does not contain a minimal generating set.

Later by $\mathbb{I}$ we denote the unit elements of $\mathbf{I}\mathbb{N}_{\infty}$.

\begin{remark}\label{remark-1.1}
We observe that the bicyclic semigroup is isomorphic to the
semigroup $\mathscr{C}_{\mathbb{N}}$ which is
generated by partial transformations $\alpha$ and $\beta$ of the set
of positive integers $\mathbb{N}$, defined as follows:
\begin{equation*}
\operatorname{dom}\alpha=\mathbb{N}, \qquad \operatorname{ran}\alpha=\mathbb{N}\setminus\{1\},  \qquad (n)\alpha=n+1
\end{equation*}
and
\begin{equation*}
\operatorname{dom}\beta=\mathbb{N}\setminus\{1\}, \qquad \operatorname{ran}\beta=\mathbb{N},  \qquad (n)\beta=n-1
\end{equation*}
(see Exercise~IV.1.11$(ii)$ in \cite{Petrich-1984}). It is obvious that $\mathbb{I}=\alpha\beta$ and $\mathscr{C}_{\mathbb{N}}$ is a submonoid  of $\mathbf{I}\mathbb{N}_{\infty}$.
\end{remark}

The semigroup of monotone (order preserving) injective partial transformations $\varphi$ of $\mathbb{N}$ such that the sets $\mathbb{N}\setminus\operatorname{dom}\varphi$ and $\mathbb{N}\setminus\operatorname{ran}\varphi$ are finite  was introduced in \cite{Gutik-Repovs-2011} and there it was denoted by $\mathscr{I}_{\infty}^{\!\nearrow}(\mathbb{N})$. Obviously, $\mathscr{I}_{\infty}^{\!\nearrow}(\mathbb{N})$ is an inverse subsemigroup of the semigroup $\mathscr{I}_\omega$. The semigroup $\mathscr{I}_{\infty}^{\!\nearrow}(\mathbb{N})$ is called \emph{the semigroup of cofinite monotone partial bijections} of $\mathbb{N}$. In \cite{Gutik-Repovs-2011} Gutik and Repov\v{s} studied properties of the semigroup $\mathscr{I}_{\infty}^{\!\nearrow}(\mathbb{N})$. In particular, they showed that $\mathscr{I}_{\infty}^{\!\nearrow}(\mathbb{N})$ is an inverse bisimple semigroup and all of its non-trivial group homomorphisms are either isomorphisms or group homomorphisms. It is obvious that $\mathbf{I}\mathbb{N}_{\infty}$ is an inverse submonoid of $\mathscr{I}_{\infty}^{\!\nearrow}(\mathbb{N})$.

%Doroshenko in \cite{Doroshenko-2005, Doroshenko-2009} studied the semigroups of endomorphisms of linearly ordered sets $\mathbb{N}$ and $\mathbb{Z}$ and their subsemigroups of cofinite endomorphisms $\mathcal{O}_{fin}(\mathbb{N})$ and $\mathcal{O}_{fin}(\mathbb{Z})$. In \cite{Doroshenko-2009} he described Green's relations, groups of automorphisms, conjugacy, centralizers of elements, growth, and free subsemigroups in these semigroups. Especially in \cite{Doroshenko-2009} it is proved that, in $\mathcal{O}_{fin}(\mathbb{N})$  the  group  of  automorphisms  consists  only  of  the  identity  mapping,  whereas  the  groups  of  automorphisms of $\mathcal{O}_{fin}(\mathbb{Z})$ is isomorphic to the semigroup of integers with operation of addition and consist only of inner automorphisms. In \cite{Doroshenko-2005} there was shown  that  both these semigroups do not admit an irreducible system of generators. In their  subsemigroups of cofinite functions all irreducible systems of generators are described there. Also, here the last semigroups are presented in terms of generators and relations.

A partial map $\alpha\colon \mathbb{N}\rightharpoonup \mathbb{N}$ is
called \emph{almost monotone} if there exists a finite subset $A$ of
$\mathbb{N}$ such that the restriction
$\alpha\mid_{\mathbb{N}\setminus A}\colon \mathbb{N}\setminus
A\rightharpoonup \mathbb{N}$ is a monotone partial map.
By $\mathscr{I}_{\infty}^{\,\Rsh\!\!\!\nearrow}(\mathbb{N})$ we
shall denote the semigroup of almost monotone
injective partial transformations of $\mathbb{N}$ such that the sets
$\mathbb{N}\setminus\operatorname{dom}\varphi$ and
$\mathbb{N}\setminus\operatorname{ran}\varphi$ are finite for all
$\varphi\in\mathscr{I}_{\infty}^{\,\Rsh\!\!\!\nearrow}(\mathbb{N})$.
Obviously, $\mathscr{I}_{\infty}^{\,\Rsh\!\!\!\nearrow}(\mathbb{N})$
is an inverse subsemigroup of the semigroup $\mathscr{I}_\omega$ and
the semigroup $\mathscr{I}_{\infty}^{\!\nearrow}(\mathbb{N})$ is an
inverse subsemigroup of
$\mathscr{I}_{\infty}^{\,\Rsh\!\!\!\nearrow}(\mathbb{N})$ too. The
semigroup $\mathscr{I}_{\infty}^{\,\Rsh\!\!\!\nearrow}(\mathbb{N})$
is called \emph{the semigroup of cofinite almost monotone injective partial
transformations} of $\mathbb{N}$.
In the paper \cite{Chuchman-Gutik-2010} the semigroup
$\mathscr{I}_{\infty}^{\,\Rsh\!\!\!\nearrow}(\mathbb{N})$ is studied. In particular, it was shown that the semigroup
$\mathscr{I}_{\infty}^{\,\Rsh\!\!\!\nearrow}(\mathbb{N})$ is inverse,
bisimple and all of its non-trivial group homomorphisms are either
isomorphisms or group homomorphisms.  In the paper \cite{Gutik-Savchuk-2019} we showed that every automorphism of a full inverse subsemigroup of $\mathscr{I}_{\infty}^{\!\nearrow}(\mathbb{N})$ which contains the semigroup $\mathscr{C}_{\mathbb{N}}$ is the identity map. Also there we  constructed a submonoid $\mathbf{I}\mathbb{N}_{\infty}^{[\underline{1}]}$ of $\mathscr{I}_{\infty}^{\,\Rsh\!\!\!\nearrow}(\mathbb{N})$ with the following property: if $S$ be an inverse subsemigroup of $\mathscr{I}_{\infty}^{\,\Rsh\!\!\!\nearrow}(\mathbb{N})$ such that $S$ contains $\mathbf{I}\mathbb{N}_{\infty}^{[\underline{1}]}$ as a submonoid, then every non-identity congruence $\mathfrak{C}$ on $S$ is a group congruence. We show that if $S$ is an inverse submonoid of $\mathscr{I}_{\infty}^{\,\Rsh\!\!\!\nearrow}(\mathbb{N})$ such that $S$ contains $\mathscr{C}_{\mathbb{N}}$ as a submonoid then $S$ is simple and the quotient semigroup $S/\mathfrak{C}_{\mathbf{mg}}$, where $\mathfrak{C}_{\mathbf{mg}}$ is minimum group congruence on $S$, is isomorphic to the additive group of integers. Also,  topologizations of inverse submonoids of $\mathscr{I}_{\infty}^{\,\Rsh\!\!\!\nearrow}(\mathbb{N})$  and embeddings of  such semigroups into compact-like topological semigroups established in \cite{Chuchman-Gutik-2010, Gutik-Savchuk-2019}. Similar results for semigroups of cofinite almost monotone partial
bijections and cofinite almost monotone partial bijections of $\mathbb{Z}$ were obtained in \cite{Gutik-Repovs-2012}.

Next we  need some notions defined in \cite{Gutik-Savchuk-2018} and \cite{Gutik-Savchuk-2019}.
For an arbitrary positive integer $n_0$ we denote $[n_0)=\left\{n\in\mathbb{N}\colon n\geqslant n_0\right\}$. Since the set of all positive integers is well ordered, the definition of the semigroup $\mathscr{I}_{\infty}^{\,\Rsh\!\!\!\nearrow}(\mathbb{N})$ implies that for every $\gamma\in\mathscr{I}_{\infty}^{\,\Rsh\!\!\!\nearrow}(\mathbb{N})$ there exists the smallest positive integer $n_{\gamma}^{\mathbf{d}}\in\operatorname{dom}\gamma$ such that the restriction $\gamma|_{\left[n_{\gamma}^{\mathbf{d}}\right)}$ of the partial map $\gamma\colon \mathbb{N}\rightharpoonup \mathbb{N}$ onto the set $\left[n_{\gamma}^{\mathbf{d}}\right)$ is an element of the semigroup $\mathscr{C}_{\mathbb{N}}$, i.e., $\gamma|_{\left[n_{\gamma}^{\mathbf{d}}\right)}$ is a some shift of $\left[n_{\gamma}^{\mathbf{d}}\right)$. For every $\gamma\in\mathscr{I}_{\infty}^{\,\Rsh\!\!\!\nearrow}(\mathbb{N})$ we put $\overrightarrow{\gamma}=\gamma|_{\left[n_{\gamma}^{\mathbf{d}}\right)}$, i.e.
\begin{equation*}
\operatorname{dom}\overrightarrow{\gamma}=\big[n_{\gamma}^{\mathbf{d}}\big), \quad (x)\overrightarrow{\gamma}=(x)\gamma \quad \hbox{for all} \; x\in \operatorname{dom}\overrightarrow{\gamma} \quad \hbox{and} \quad \operatorname{ran}\overrightarrow{\gamma}=\left(\operatorname{dom}\overrightarrow{\gamma}\right)\gamma.
\end{equation*}
Also, we put
\begin{equation*}
\underline{n}_{\gamma}^{\mathbf{d}}=\min\operatorname{dom}\gamma \qquad \hbox{for} \quad \gamma\in\mathscr{I}_{\infty}^{\,\Rsh\!\!\!\nearrow}(\mathbb{N}).
\end{equation*}
It is obvious that $\underline{n}_{\gamma}^{\mathbf{d}}= n_{\gamma}^{\mathbf{d}}$ when $\gamma\in\mathscr{C}_{\mathbb{N}}$, and $\underline{n}_{\gamma}^{\mathbf{d}}< n_{\gamma}^{\mathbf{d}}$ when $\gamma\in\mathscr{I}_{\infty}^{\,\Rsh\!\!\!\nearrow}(\mathbb{N})\setminus\mathscr{C}_{\mathbb{N}}$. Also for any $\gamma\in\mathbf{I}\mathbb{N}_{\infty}$ we denote
\begin{equation*}
  \underline{n}_{\gamma}^{\mathbf{r}}=(\underline{n}_{\gamma}^{\mathbf{d}})\gamma \qquad \hbox{and} \qquad n_{\gamma}^{\mathbf{r}}=(n_{\gamma}^{\mathbf{d}})\gamma.
\end{equation*}

The results of Section~3 of \cite{Gutik-Savchuk-2020} imply that $n_{\gamma}^{\mathbf{r}}-\underline{n}_{\gamma}^{\mathbf{r}}=n_{\gamma}^{\mathbf{d}}-\underline{n}_{\gamma}^{\mathbf{d}}$ for any $\gamma\in\mathbf{I}\mathbb{N}_{\infty}$, and moreover for any non-negative integer $j$
\begin{equation*}
  \mathbf{I}\mathbb{N}_{\infty}^{\boldsymbol{g}[j]}=\big\{\gamma\in \mathbf{I}\mathbb{N}_{\infty}\colon n_{\gamma}^{\mathbf{d}}-\underline{n}_{\gamma}^{\mathbf{d}}\leqslant j\big\}
\end{equation*}
is a simple  inverse subsemigroup of $\mathbf{I}\mathbb{N}_{\infty}$ such that $\mathbf{I}\mathbb{N}_{\infty}$ admits the following infinite semigroup series
\begin{equation*}
\mathscr{C}_{\mathbb{N}}=\mathbf{I}\mathbb{N}_{\infty}^{\boldsymbol{g}[0]}= \mathbf{I}\mathbb{N}_{\infty}^{\boldsymbol{g}[1]}\subsetneqq \mathbf{I}\mathbb{N}_{\infty}^{\boldsymbol{g}[2]}\subsetneqq \mathbf{I}\mathbb{N}_{\infty}^{\boldsymbol{g}[3]}\subsetneqq \cdots \subsetneqq \mathbf{I}\mathbb{N}_{\infty}^{\boldsymbol{g}[k]}\subsetneqq \cdots \subset \mathbf{I}\mathbb{N}_{\infty}.
\end{equation*}

For any positive integer $k$ the semigroup $\mathbf{I}\mathbb{N}_{\infty}^{\boldsymbol{g}[k]}$ is called the \emph{monoid of cofinite isometries of positive integers with the noise} $k$.

A (\emph{semi})\emph{topological} \emph{semigroup} is a topological space with a (separately) continuous semigroup operation. An inverse topological semigroup with continuous inversion is called a \emph{topological inverse semigroup}.

A topology $\tau$ on a semigroup $S$ is called:
\begin{itemize}
  \item a \emph{semigroup} topology if $(S,\tau)$ is a topological semigroup;
   \item an \emph{inverse semigroup} topology if $(S,\tau)$ is a topological inverse semigroup;
  \item a \emph{shift-continuous} topology if $(S,\tau)$ is a semitopological semigroup.
\end{itemize}

The bicyclic monoid admits only the discrete semigroup Hausdorff topology \cite{Eberhart-Selden-1969}. Bertman and  West in \cite{Bertman-West-1976} extended this result for the case of Hausdorff semitopological semigroups. Stable and $\Gamma$-compact topological semigroups do not contain the bicyclic monoid~\cite{Anderson-Hunter-Koch-1965, Hildebrant-Koch-1986, Koch-Wallace-1957}. The problem of embedding the bicyclic monoid into compact-like topological semigroups was studied in \cite{Banakh-Dimitrova-Gutik-2009, Banakh-Dimitrova-Gutik-2010, Bardyla-Ravsky-2019, Gutik-Repovs-2007}.

In this paper we study algebraic properties of the monoid $\mathbf{I}\mathbb{N}_{\infty}^{\boldsymbol{g}[j]}$ and extend results of the papers \cite{Eberhart-Selden-1969} and \cite{Bertman-West-1976} to the semigroups $\mathbf{I}\mathbb{N}_{\infty}^{\boldsymbol{g}[j]}$, $j\geqslant 0$. In particular we show that for any positive integer $j$ every Hausdorff shift-continuous  topology $\tau$ on $\mathbf{I}\mathbb{N}_{\infty}^{\boldsymbol{g}[j]}$ is discrete and and if $\mathbf{I}\mathbb{N}_{\infty}^{\boldsymbol{g}[j]}$ is a proper dense subsemigroup of a Hausdorff semitopological semigroup $S$, then $S\setminus \mathbf{I}\mathbb{N}_{\infty}^{\boldsymbol{g}[j]}$  is a closed ideal of $S$, and moreover if $S$ is a topological inverse semigroup then $S\setminus \mathbf{I}\mathbb{N}_{\infty}^{\boldsymbol{g}[j]}$ is a topological group. Also we describe the algebraic and topological structure of the closure of the monoid $\mathbf{I}\mathbb{N}_{\infty}^{\boldsymbol{g}[j]}$ in a locally compact topological inverse semigroup.

Latter in this paper without loss of generality we may assume that $j$ is an arbitrary positive integer $\geqslant 2$.

\smallskip

\section{Algebraic properties of the monoid $\mathbf{I}\mathbb{N}_{\infty}^{\boldsymbol{g}[j]}$}\label{sec-2}

The following simple proposition describes Green's relations on the monoid $\mathbf{I}\mathbb{N}_{\infty}^{\boldsymbol{g}[j]}$.

\begin{proposition}\label{proposition-1.1}
For elements $\gamma$ and $\delta$ of the semigroup $\mathbf{I}\mathbb{N}_{\infty}^{\boldsymbol{g}[j]}$ the following statements hold:
\begin{itemize}
  \item[$(i)$] $\gamma\mathscr{L}\delta$ in $\mathbf{I}\mathbb{N}_{\infty}^{\boldsymbol{g}[j]}$ if and only if $\operatorname{dom}\gamma=\operatorname{dom}\delta$;
  \item[$(ii)$] $\gamma\mathscr{R}\delta$ in $\mathbf{I}\mathbb{N}_{\infty}^{\boldsymbol{g}[j]}$ if and only if $\operatorname{ran}\gamma=\operatorname{ran}\delta$;
  \item[$(iii)$] $\gamma\mathscr{H}\delta$ in $\mathbf{I}\mathbb{N}_{\infty}^{\boldsymbol{g}[j]}$ if and only if $\gamma=\delta$;
  \item[$(iv)$] $\gamma\mathscr{D}\delta$ in $\mathbf{I}\mathbb{N}_{\infty}^{\boldsymbol{g}[j]}$ if and only if $\operatorname{dom}\gamma$ $(\operatorname{ran}\gamma)$ and  $\operatorname{dom}\delta$ $(\operatorname{ran}\delta)$ are isometric subsets of $\mathbb{N}$, i.e., there exists an isometry from $\operatorname{dom}\gamma$ $(\operatorname{ran}\gamma)$ onto $\operatorname{dom}\delta$ $(\operatorname{ran}\delta)$;
  \item[$(v)$] $\gamma\mathscr{J}\delta$ in $\mathbf{I}\mathbb{N}_{\infty}^{\boldsymbol{g}[j]}$, i.e., $\mathbf{I}\mathbb{N}_{\infty}^{\boldsymbol{g}[j]}$ is a simple semigroup.
\end{itemize}
\end{proposition}

\begin{proof}
Statements $(i)$, $(ii)$ and $(iii)$ immediately follow from Proposition~3.2.11 of \cite{Lawson-1998} and corresponding statements of Proposition~1 of \cite{Gutik-Savchuk-2018}.

Statement $(iv)$ follows from the definition of the monoid and Proposition~3.2.5 of \cite{Lawson-1998}.

Statement $(v)$ follows from Theorem~5 of \cite{Gutik-Savchuk-2019}.
\end{proof}

Proposition \ref{proposition-1.2} follows from the definition of the natural partial order $\preccurlyeq$ on an inverse semigroup and the statement that every element of the monoid $\mathbf{I}\mathbb{N}_{\infty}^{\boldsymbol{g}[j]}$ is a partial shift of the integers (see \cite[Lemma~1]{Gutik-Savchuk-2018}).

\begin{proposition}\label{proposition-1.2}
Let $\gamma$ and $\delta$  be elements of the monoid $\mathbf{I}\mathbb{N}_{\infty}^{\boldsymbol{g}[j]}$. Then the following conditions are equivalent:
\begin{itemize}
  \item[$(i)$] $\gamma\preccurlyeq\delta$  in $\mathbf{I}\mathbb{N}_{\infty}^{\boldsymbol{g}[j]}$
  \item[$(ii)$] $n_{\gamma}^{\mathbf{r}}-n_{\gamma}^{\mathbf{d}}=n_{\delta}^{\mathbf{r}}-n_{\delta}^{\mathbf{d}}$ and $\operatorname{dom}\gamma\subseteq \operatorname{dom}\delta$;
  \item[$(iii)$] $n_{\gamma}^{\mathbf{r}}-n_{\gamma}^{\mathbf{d}}=n_{\delta}^{\mathbf{r}}-n_{\delta}^{\mathbf{d}}$ and $\operatorname{ran}\gamma\subseteq \operatorname{ran}\delta$.
\end{itemize}
\end{proposition}

It is obvious that in statements $(ii)$ and $(iii)$ of Proposition~\ref{proposition-1.2} we may replace the symbols $n_{\gamma}^{\mathbf{r}}$ and $n_{\gamma}^{\mathbf{d}}$ by $\underline{n}_{\gamma}^{\mathbf{r}}$ and $\underline{n}_{\gamma}^{\mathbf{d}}$, respectively.

The definition of the minimum group congruence $\mathfrak{C}_{\mathbf{mg}}$ on $\mathbf{I}\mathbb{N}_{\infty}^{\boldsymbol{g}[j]}$ and Proposition~\ref{proposition-1.2} imply the following proposition.

\begin{proposition}\label{proposition-1.3}
Let $\gamma$ and $\delta$  be elements of the monoid $\mathbf{I}\mathbb{N}_{\infty}^{\boldsymbol{g}[j]}$. Then
$\gamma\mathfrak{C}_{\mathbf{mg}}\delta$  in $\mathbf{I}\mathbb{N}_{\infty}^{\boldsymbol{g}[j]}$ if and only if
$n_{\gamma}^{\mathbf{r}}-n_{\gamma}^{\mathbf{d}}=n_{\delta}^{\mathbf{r}}-n_{\delta}^{\mathbf{d}}$.
Moreover,  the quotient semigroup $\mathbf{I}\mathbb{N}_{\infty}^{\boldsymbol{g}[j]}/\mathfrak{C}_{\mathbf{mg}}$ is isomorphic to the additive group of integers $\mathbb{Z}(+)$ by the map
\begin{equation*}
\pi_{\mathfrak{C}_{\mathbf{mg}}}\colon \mathbf{I}\mathbb{N}_{\infty}^{\boldsymbol{g}[j]}\to \mathbb{Z}(+), \quad \gamma\mapsto n_{\delta}^{\mathbf{r}}-n_{\delta}^{\mathbf{d}}.
\end{equation*}
\end{proposition}

\begin{example}\label{example-1.4}
We put $\mathcal{C}\mathbf{I}\mathbb{N}_{\infty}^{\boldsymbol{g}[j]}=\mathbf{I}\mathbb{N}_{\infty}^{\boldsymbol{g}[j]}\sqcup\mathbb{Z}(+)$ and extend the multiplications from $\mathbf{I}\mathbb{N}_{\infty}^{\boldsymbol{g}[j]}$ and $\mathbb{Z}(+)$ onto $\mathcal{C}\mathbf{I}\mathbb{N}_{\infty}^{\boldsymbol{g}[j]}$ in the following way:
\begin{equation*}
  k\cdot\gamma=\gamma\cdot k=k+(\gamma)\pi_{\mathfrak{C}_{\mathbf{mg}}}\in \mathbb{Z}(+), \qquad \hbox{for all} \quad k\in\mathbb{Z}(+) \quad \hbox{and} \quad \gamma\in \mathbf{I}\mathbb{N}_{\infty}^{\boldsymbol{g}[j]}.
\end{equation*}
By Theorem~2.17 from \cite[Vol.~1, p.~77]{Carruth-Hildebrant-Koch-1983-1986} so defined binary operation is a semigroup operation on $\mathcal{C}\mathbf{I}\mathbb{N}_{\infty}^{\boldsymbol{g}[j]}$ such that $\mathbb{Z}(+)$ is an ideal in $\mathcal{C}\mathbf{I}\mathbb{N}_{\infty}^{\boldsymbol{g}[j]}$. Also,  this semigroup operation extends the  natural partial order $\preccurlyeq$ from $\mathbf{I}\mathbb{N}_{\infty}^{\boldsymbol{g}[j]}$ onto $\mathcal{C}\mathbf{I}\mathbb{N}_{\infty}^{\boldsymbol{g}[j]}$ in the following way:
\begin{itemize}
  \item[$(i)$] all distinct elements of $\mathbb{Z}(+)$ are pair-wise incomparable;
  \item[$(ii)$] $k\preccurlyeq\gamma$ if and only if $n_{\gamma}^{\mathbf{r}}-n_{\gamma}^{\mathbf{d}}=k$ for $k\in \mathbb{Z}(+)$ and $\gamma\in \mathbf{I}\mathbb{N}_{\infty}^{\boldsymbol{g}[j]}$.
\end{itemize}
For any $x\in \mathcal{C}\mathbf{I}\mathbb{N}_{\infty}^{\boldsymbol{g}[j]}$ we denote ${\uparrow_{\preccurlyeq}}x=\big\{y\in \mathcal{C}\mathbf{I}\mathbb{N}_{\infty}^{\boldsymbol{g}[j]}\colon x\preccurlyeq y\big\}$.
\end{example}

By Proposition~7 of \cite{Gutik-Savchuk-2018} the map $\mathfrak{h}\colon \mathbf{I}\mathbb{N}_{\infty}\to \mathscr{C}_{\mathbb{N}}$, $\gamma\mapsto \overrightarrow{\gamma}$ is a homomorphism. Then its restriction $\mathfrak{h}|_{\mathbf{I}\mathbb{N}_{\infty}^{\boldsymbol{g}[j]}}\colon \mathbf{I}\mathbb{N}_{\infty}^{\boldsymbol{g}[j]}\to \mathscr{C}_{\mathbb{N}}$ is homomorphism, too.

A \emph{homomorphic retraction} of a semigroup $S$ is a map from $S$ into $S$ which is both a retraction and a homomorphism. The image of the homomorphic retraction is called a \emph{homomorphic retract}. These terms seem to have first appeared in \cite{Brown-1965}.

Since $(\gamma)\mathfrak{h}=\overrightarrow{\gamma}=\gamma$ for any $\gamma\in \mathscr{C}_{\mathbb{N}}$ we get the following proposition.

\begin{proposition}\label{proposition-1.5}
The map $\mathfrak{h}\colon \mathbf{I}\mathbb{N}_{\infty}^{\boldsymbol{g}[j]}\to \mathscr{C}_{\mathbb{N}}$, $\gamma\mapsto \overrightarrow{\gamma}$ is a homomorphic retraction, and hence the monoid $\mathscr{C}_{\mathbb{N}}$ is a homomorphic retract of $\mathbf{I}\mathbb{N}_{\infty}^{\boldsymbol{g}[j]}$.
\end{proposition}

For any subset $M\subseteq \{2,\ldots,j\}$ we denote
\begin{equation*}
  \mathbf{I}\mathbb{N}_{\infty}^{\boldsymbol{g}[j]}[M]=\big\{\gamma\in \mathbf{I}\mathbb{N}_{\infty}^{\boldsymbol{g}[j]}\colon   n_{\gamma}^{\mathbf{d}}-x\in M\cup\{0\} \: \hbox{for all} \: x\in \operatorname{dom}\gamma \: \hbox{such that} \: x\leqslant n_{\gamma}^{\mathbf{d}}\big\}.
\end{equation*}

For arbitrary $M_1,M_2\subseteq \{2,\ldots,j\}$ it is obvious that $\mathbf{I}\mathbb{N}_{\infty}^{\boldsymbol{g}[j]}[M_1]\subseteq \mathbf{I}\mathbb{N}_{\infty}^{\boldsymbol{g}[j]}[M_2]$ if and only if $M_1\subseteq M_2$, and moreover we have that $\mathbf{I}\mathbb{N}_{\infty}^{\boldsymbol{g}[j]}[M]=\mathscr{C}_{\mathbb{N}}$ when $M=\varnothing$ and $\mathbf{I}\mathbb{N}_{\infty}^{\boldsymbol{g}[j]}[M]=\mathbf{I}\mathbb{N}_{\infty}^{\boldsymbol{g}[j]}$ when $M=\{2,\ldots,j\}$.

\begin{remark}\label{remark-1.6}
By Lemma~1 of \cite{Gutik-Savchuk-2018} we get that
\begin{equation*}
  \mathbf{I}\mathbb{N}_{\infty}^{\boldsymbol{g}[j]}[M]=\big\{\gamma\in \mathbf{I}\mathbb{N}_{\infty}^{\boldsymbol{g}[j]}\colon   n_{\gamma}^{\mathbf{r}}-x\in M\cup\{0\} \: \hbox{for all} \: x\in \operatorname{ran}\gamma \: \hbox{such that} \: x\leqslant n_{\gamma}^{\mathbf{r}}\big\}.
\end{equation*}
\end{remark}

\begin{proposition}\label{proposition-1.7}
$\mathbf{I}\mathbb{N}_{\infty}^{\boldsymbol{g}[j]}[M]$ is an inverse semigroup of $\mathbf{I}\mathbb{N}_{\infty}^{\boldsymbol{g}[j]}$ for any $M\subseteq \{2,\ldots,j\}$.
\end{proposition}

\begin{proof}
Fix any $\gamma,\delta\in\mathbf{I}\mathbb{N}_{\infty}^{\boldsymbol{g}[j]}[M]$. We consider the following cases.
\begin{enumerate}
  \item If $n_{\gamma}^{\mathbf{r}}\leqslant n_{\delta}^{\mathbf{d}}$ then $n_{\gamma\delta}^{\mathbf{r}}=n_{\delta}^{\mathbf{r}}$ and $\operatorname{ran}(\gamma\delta)\subseteq \operatorname{ran}\delta$, because by Lemma~1 from \cite{Gutik-Savchuk-2018} all elements of $\mathbf{I}\mathbb{N}_{\infty}$ are partial shifts of the set $\mathbb{N}$. This and Remark~\ref{remark-1.6} imply that $\gamma\delta\in\mathbf{I}\mathbb{N}_{\infty}^{\boldsymbol{g}[j]}[M]$.
  \item If $n_{\gamma}^{\mathbf{r}}> n_{\delta}^{\mathbf{d}}$ then by similar arguments as in the previous case we get that $n_{\gamma\delta}^{\mathbf{d}}=n_{\gamma}^{\mathbf{d}}$ and $\operatorname{dom}(\gamma\delta)\subseteq \operatorname{dom}\delta$. This implies that $\gamma\delta\in\mathbf{I}\mathbb{N}_{\infty}^{\boldsymbol{g}[j]}[M]$.
\end{enumerate}

Remark~\ref{remark-1.6} implies that if $\gamma\in\mathbf{I}\mathbb{N}_{\infty}^{\boldsymbol{g}[j]}[M]$ then so is $\gamma^{-1}$.
\end{proof}
%%%%%%%%%%%%%%%%%%%%%%%%%%%%%%%%%%%%%%%%%%%%%%%%%%%%%%%%%%%%

\section{On a topologization and a closure of the monoid $\mathbf{I}\mathbb{N}_{\infty}^{\boldsymbol{g}[j]}$}\label{sec-3}

Later in the paper by $\mathbb{I}$ we denote the identity map of $\mathbb{N}$, and assume that $\alpha$ and $\beta$ are the elements of the submonoid $\mathscr{C}_{\mathbb{N}}$ in $\mathbf{I}\mathbb{N}_{\infty}$ which are defined in Remark~\ref{remark-1.1}.

It is obvious that $\alpha\beta=\mathbb{I}$ and $\beta\alpha$ is the identity map of $\mathbb{N}\setminus\{1\}$. This implies the following lemma.

\begin{lemma}\label{lemma-2.1}
If $\gamma\in \mathbf{I}\mathbb{N}_{\infty}$, then
\begin{itemize}
  \item[$(i)$] $\beta\alpha\cdot\gamma=\gamma$ if and only if $\operatorname{dom}\gamma\subseteq \mathbb{N}\setminus\{1\}$;
  \item[$(ii)$] $\gamma\cdot\beta\alpha=\gamma$ if and only if $\operatorname{ran}\gamma\subseteq \mathbb{N}\setminus\{1\}$.
\end{itemize}
\end{lemma}

For any positive integer $i$ let $\varepsilon^{[i]}$ be the identity map of the set $\mathbb{N}\setminus\{i\}$.

The following theorem generalized the results on the topologizabily of the bicyclic monoid obtained in \cite{Eberhart-Selden-1969} and \cite{Bertman-West-1976}.

\begin{theorem}\label{theorem-2.2}
For any positive integer $j$ every Hausdorff shift-continuous  topology $\tau$ on $\mathbf{I}\mathbb{N}_{\infty}^{\boldsymbol{g}[j]}$ is discrete.
\end{theorem}

\begin{proof}
Since $\tau$ is Hausdorff, every retract of $\big(\mathbf{I}\mathbb{N}_{\infty}^{\boldsymbol{g}[j]},\tau\big)$ is its closed subset. It is obvious  that $\beta\alpha \cdot\mathbf{I}\mathbb{N}_{\infty}^{\boldsymbol{g}[j]}$ and $\mathbf{I}\mathbb{N}_{\infty}^{\boldsymbol{g}[j]}\cdot\beta\alpha$ are retracts of the topological space $\big(\mathbf{I}\mathbb{N}_{\infty}^{\boldsymbol{g}[j]},\tau\big)$, because $\beta\alpha$ is an idempotent of $\mathbf{I}\mathbb{N}_{\infty}^{\boldsymbol{g}[j]}$. Later we shall show that the set $\mathbf{I}\mathbb{N}_{\infty}^{\boldsymbol{g}[j]}\setminus\big(\beta\alpha\cdot \mathbf{I}\mathbb{N}_{\infty}^{\boldsymbol{g}[j]}\cup \mathbf{I}\mathbb{N}_{\infty}^{\boldsymbol{g}[j]}\cdot\beta\alpha\big)$ is finite.

By Lemma~\ref{lemma-2.1}, $\gamma\in \mathbf{I}\mathbb{N}_{\infty}^{\boldsymbol{g}[j]}\setminus\big(\beta\alpha\cdot \mathbf{I}\mathbb{N}_{\infty}^{\boldsymbol{g}[j]}\cup \mathbf{I}\mathbb{N}_{\infty}^{\boldsymbol{g}[j]}\cdot\beta\alpha\big)$ if and only if $1\in\operatorname{dom}\gamma$, $1\in\operatorname{ran}\gamma$, and $n_{\gamma}^{\mathbf{d}}-\underline{n}_{\gamma}^{\mathbf{d}}\leqslant j$. Then by Lemma~1 of \cite{Gutik-Savchuk-2018}, $\gamma$ is a partial shift of the set of integers, and hence $\gamma$ is an idempotent of $\mathbf{I}\mathbb{N}_{\infty}^{\boldsymbol{g}[j]}$ such that $1\in\operatorname{dom}\gamma$ and $\varepsilon^{[2]}\cdot \ldots\cdot \varepsilon^{[j-1]}\preccurlyeq \gamma$. It is obvious that such idempotents $\gamma$ are finitely many in $\mathbf{I}\mathbb{N}_{\infty}^{\boldsymbol{g}[j]}$, and hence the set $\mathbf{I}\mathbb{N}_{\infty}^{\boldsymbol{g}[j]}\setminus\big(\beta\alpha\cdot \mathbf{I}\mathbb{N}_{\infty}^{\boldsymbol{g}[j]}\cup \mathbf{I}\mathbb{N}_{\infty}^{\boldsymbol{g}[j]}\cdot\beta\alpha\big)$ is finite. This implies that the point $\mathbb{I}$ has a finite open neighbourhood and hence $\mathbb{I}$ is an isolated point of the topological space $\big(\mathbf{I}\mathbb{N}_{\infty}^{\boldsymbol{g}[j]},\tau\big)$.

We observe that $\mathbf{I}\mathbb{N}_{\infty}$, and hence $\mathbf{I}\mathbb{N}_{\infty}^{\boldsymbol{g}[j]}$, is a submonoid of the  semigroup $\mathscr{I}_{\infty}^{\!\nearrow}(\mathbb{N})$ of cofinite monotone partial bijections of $\mathbb{N}$ \cite{Gutik-Savchuk-2018}. By Proposition~2.2 of \cite{Gutik-Repovs-2011} every right translation and every left translation by an element of the semigroup
$\mathscr{I}_{\infty}^{\!\nearrow}(\mathbb{N})$  is a finite-to-one map, and hence such conditions hold for the semigroup $\mathbf{I}\mathbb{N}_{\infty}^{\boldsymbol{g}[j]}$. Also by Theorem~5 of \cite{Gutik-Savchuk-2019}, $\mathbf{I}\mathbb{N}_{\infty}^{\boldsymbol{g}[j]}$ is a simple semigroup. This implies that for any $\chi\in \mathbf{I}\mathbb{N}_{\infty}^{\boldsymbol{g}[j]}$ there exist $\alpha,\beta\in\mathbf{I}\mathbb{N}_{\infty}^{\boldsymbol{g}[j]}$ such that $\alpha\chi\beta=\mathbb{I}$, and moreover the equality $\alpha\chi\beta=\mathbb{I}$ has finitely many solutions. Since $\mathbb{I}$ is an isolated point of  $\big(\mathbf{I}\mathbb{N}_{\infty}^{\boldsymbol{g}[j]},\tau\big)$, the separate continuity of the semigroup operation in $\big(\mathbf{I}\mathbb{N}_{\infty}^{\boldsymbol{g}[j]},\tau\big)$ and the above arguments imply that $\big(\mathbf{I}\mathbb{N}_{\infty}^{\boldsymbol{g}[j]},\tau\big)$ is the discrete space.
\end{proof}

The following proposition generalized results obtained for the bicyclic monoid in \cite{Eberhart-Selden-1969} and \cite{Gutik-2015}.

\begin{proposition}\label{proposition-2.3}
Let $j$ be any positive integer and $\mathbf{I}\mathbb{N}_{\infty}^{\boldsymbol{g}[j]}$ be a proper dense subsemigroup of a Hausdorff semitopological semigroup $S$. Then $I=S\setminus \mathbf{I}\mathbb{N}_{\infty}^{\boldsymbol{g}[j]}$  is a closed ideal of $S$.
\end{proposition}

\begin{proof}
By Theorem~\ref{theorem-2.2}, $\mathbf{I}\mathbb{N}_{\infty}^{\boldsymbol{g}[j]}$ is a discrete subspace of $S$, and hence by Lemma~3 of \cite{Gutik-Savchuk-2017}, $\mathbf{I}\mathbb{N}_{\infty}^{\boldsymbol{g}[j]}$ is an open subspace of $S$.

Fix an arbitrary element $y\in I$. If $xy=z\notin I$ for some $x\in\mathbf{I}\mathbb{N}_{\infty}^{\boldsymbol{g}[j]}$ then there exists an open neighbourhood $U(y)$ of the point $y$ in the space $S$ such that $\{x\}\cdot U(y)=\{z\}\subset\mathbf{I}\mathbb{N}_{\infty}^{\boldsymbol{g}[j]}$. The neighbourhood $U(y)$ contains infinitely many elements of the semigroup $\mathbf{I}\mathbb{N}_{\infty}^{\boldsymbol{g}[j]}$. This contradicts Proposition~2.2 of \cite{Gutik-Repovs-2011}, which states that for each $v,w\in \mathbf{I}\mathbb{N}_{\infty}^{\boldsymbol{g}[j]}$ both sets $\big\{u\in \mathbf{I}\mathbb{N}_{\infty}^{\boldsymbol{g}[j]} \colon vu=w\big\}$ and $\big\{u\in \mathbf{I}\mathbb{N}_{\infty}^{\boldsymbol{g}[j]} \colon uv=w\big\}$ are finite. The obtained contradiction implies that $xy\in I$ for all $x\in  \mathbf{I}\mathbb{N}_{\infty}^{\boldsymbol{g}[j]}$ and $y\in I$. The proof of the statement that $yx\in I$ for all $x\in \mathbf{I}\mathbb{N}_{\infty}^{\boldsymbol{g}[j]}$ and $y\in I$ is similar.

Suppose to the contrary that $xy=w\notin I$ for some $x,y\in I$. Then $w\in \mathbf{I}\mathbb{N}_{\infty}^{\boldsymbol{g}[j]}$ and the separate continuity of the semigroup operation in $S$ implies that there exist open neighbourhoods $U(x)$ and $U(y)$ of the points $x$ and $y$ in $S$, respectively, such that $\{x\}\cdot U(y)=\{w\}$ and $U(x)\cdot \{y\}=\{w\}$. Since both neighbourhoods $U(x)$ and $U(y)$ contain infinitely many elements of the semigroup $\mathbf{I}\mathbb{N}_{\infty}^{\boldsymbol{g}[j]}$, both equalities $\{x\}\cdot U(y)=\{w\}$ and $U(x)\cdot \{y\}=\{w\}$ contradict mentioned above Proposition~2.2 from \cite{Gutik-Repovs-2011}. The obtained contradiction implies that $xy\in I$.
\end{proof}

\begin{lemma}\label{lemma-2.4}
Let $j$ be any positive integer $\geqslant 2$. Then the element
$
  \varepsilon\cdot (\beta\varepsilon)^j\cdot\alpha^j
$
is an idempotent of the submonoid $\mathscr{C}_{\mathbb{N}}$ for any idempotent $\varepsilon$ of the monoid $\mathbf{I}\mathbb{N}_{\infty}^{\boldsymbol{g}[j]}$.
\end{lemma}

\begin{proof}
Since $\mathbb{I}=\alpha\beta$, we have that
\begin{equation*}
  \varepsilon\cdot\beta\varepsilon\alpha\cdot \beta^2\varepsilon\alpha^2\cdot\ldots\cdot \beta^j\varepsilon\alpha^j= \varepsilon\cdot (\mathbb{I}\beta\varepsilon)^j\cdot\alpha^j=\varepsilon\cdot (\beta\varepsilon)^j\cdot\alpha^j
\end{equation*}
and
\begin{equation*}
  \beta^k\varepsilon\alpha^k\cdot \beta^k\varepsilon\alpha^k=\beta^k\varepsilon\mathbb{I}\varepsilon\alpha^k=\beta^k\varepsilon\varepsilon\alpha^k=\beta^k\varepsilon\alpha^k,
\end{equation*}
for any positive integer $k$. Also, $\varepsilon(\beta\varepsilon)^j\alpha^j$ is an idempotent of $\mathbf{I}\mathbb{N}_{\infty}^{\boldsymbol{g}[j]}$, because $\mathbf{I}\mathbb{N}_{\infty}^{\boldsymbol{g}[j]}$ is an inverse semigroup
and the product of idempotents in an inverse semigroup is an idempotent as well.

By definitions of the partial transformations $\alpha$ and $\beta$ and the above part of the proof we get that
\begin{equation}\label{eq-2.1}
  n_{\beta^k\varepsilon\alpha^k}^{\mathbf{d}}=n_{\varepsilon}^{\mathbf{d}}+k \qquad \hbox{and} \qquad
  \underline{n}_{\beta^k\varepsilon\alpha^k}^{\mathbf{d}}=\underline{n}_{\varepsilon}^{\mathbf{d}}+k,
\end{equation}
and hence
\begin{equation}\label{eq-2.2}
  n_{\beta^k\varepsilon\alpha^k}^{\mathbf{d}}-\underline{n}_{\beta^k\varepsilon\alpha^k}^{\mathbf{d}}=
  n_{\varepsilon}^{\mathbf{d}}-\underline{n}_{\varepsilon}^{\mathbf{d}},
\end{equation}
for any positive integer $k$. Then equalities  \eqref{eq-2.1} and \eqref{eq-2.2} imply that for any $k=1,\ldots,j$ the idempotent
\begin{equation*}
\varepsilon_k=\varepsilon(\beta\varepsilon)^k\alpha^k
\end{equation*}
has the following properties:
\begin{equation*}
  n_{\varepsilon_k}^{\mathbf{d}}=n_{\beta^k\varepsilon\alpha^k}^{\mathbf{d}}, \quad \underline{n}_{\varepsilon_k}^{\mathbf{d}}=\underline{n}_{\beta^k\varepsilon\alpha^k}^{\mathbf{d}},
\end{equation*}
and
\begin{equation*}
    1,\ldots,\underline{n}_{\varepsilon}^{\mathbf{d}},\ldots,\underline{n}_{\varepsilon}^{\mathbf{d}}+k-1, n_{\varepsilon}^{\mathbf{d}}-1,n_{\varepsilon}^{\mathbf{d}},\ldots,n_{\varepsilon}^{\mathbf{d}}+k-1\notin\operatorname{dom}\varepsilon_k.
\end{equation*}
Hence we get that $\varepsilon_j$ is the identity map of $\left[n_{\varepsilon}^{\mathbf{d}}+j\right)$, which implies the statement of the lemma. \end{proof}

\begin{lemma}\label{lemma-2.5}
Let $j$ be any positive integer and $\mathbf{I}\mathbb{N}_{\infty}^{\boldsymbol{g}[j]}$ be a proper dense subsemigroup of a Hausdorff topological inverse semigroup $S$. Then there exists an idempotent $e\in S\setminus \mathbf{I}\mathbb{N}_{\infty}^{\boldsymbol{g}[j]}$ such that $V(e)\cap E(\mathscr{C}_\mathbb{N})$ is an infinite subset for any open neighbourhood $V(e)$ of $e$ in $S$.
\end{lemma}

\begin{proof}
By Proposition~\ref{proposition-2.3}, $S\setminus \mathbf{I}\mathbb{N}_{\infty}^{\boldsymbol{g}[j]}$ is an ideal of $S$. Since $S$ is an inverse semigroup, $S\setminus \mathbf{I}\mathbb{N}_{\infty}^{\boldsymbol{g}[j]}$ contains an idempotent.

Put $f$ be an arbitrary idempotent of $S\setminus \mathbf{I}\mathbb{N}_{\infty}^{\boldsymbol{g}[j]}$. Since the unit element of a Hausdorff topological monoid is again the unit element of its closure in a topological semigroup, for an arbitrary positive integer $k$ by Proposition~\ref{proposition-2.3} we have that
\begin{equation*}
\beta^kf\alpha^k\cdot \beta^kf\alpha^k=\beta^kf\mathbb{I} f\alpha^k=\beta^kff\alpha^k=\beta^kf\alpha^k,
\end{equation*}
and hence $\beta^kf\alpha^k\in E(S)\setminus E\big(\mathbf{I}\mathbb{N}_{\infty}^{\boldsymbol{g}[j]}\big)$. This implies that  $e=f\cdot \beta f\alpha\cdot\ldots\cdot\beta^j f\alpha^j$ is an idempotent in $S$ because $S$ is an inverse semigroup. The continuity of the semigroup operation in $S$ implies that for every open neighbourhood $V(e)$ of the point $e$ in $S$ there exists an open neighbourhood $W(f)$ of the point $f$ in $S$ such that
\begin{equation*}
  W(f)\cdot \beta\cdot W(f)\alpha\cdot\ldots\cdot\beta^j\cdot W(f)\cdot\alpha^j\subseteq V(e).
\end{equation*}
By Proposition~II.3 of \cite{Eberhart-Selden-1969} the set $W(f)\cap E\big(\mathbf{I}\mathbb{N}_{\infty}^{\boldsymbol{g}[j]}\big)$ is infinite. Since for any positive integer $n_0$ there exist finitely many idempotents $\varepsilon\in \mathbf{I}\mathbb{N}_{\infty}^{\boldsymbol{g}[j]}$ such that $n_{\varepsilon}^{\mathbf{d}}=n_0$, we conclude that the set
$ %\begin{equation*}
  \big\{n_{\varepsilon}^{\mathbf{d}}\colon \varepsilon\in W(f)\cap E\big(\mathbf{I}\mathbb{N}_{\infty}^{\boldsymbol{g}[j]}\big)\big\}
$ %\end{equation*}
is infinite, too. Then there exists an infinite sequence $\left\{\varphi_i\right\}_{i\in\mathbb{N}}$ of idempotents of $W(f)\cap E\big(\mathbf{I}\mathbb{N}_{\infty}^{\boldsymbol{g}[j]}\big)$ such that $n_{\varphi_{i_1}}^{\mathbf{d}}\neq n_{\varphi_{i_2}}^{\mathbf{d}}$ for any distinct positive integers $i_1$ and  $i_2$. Lemma~\ref{lemma-2.5} implies that
$
  \varphi_i\cdot (\beta\varphi_i)^j\cdot\alpha^j
$
is an idempotent of the submonoid $\mathscr{C}_{\mathbb{N}}$ which belongs to $V(e)$ for any positive integer $i$. Since the set
$\big\{n_{\varepsilon}^{\mathbf{d}}\colon \varepsilon\in W(f)\cap E\big(\mathbf{I}\mathbb{N}_{\infty}^{\boldsymbol{g}[j]}\big)\big\}$ is infinite, the set $V(e)\cap E(\mathscr{C}_\mathbb{N})$ is infinite, too.
\end{proof}

\begin{theorem}\label{theorem-2.6}
Let $j$ be any positive integer and $\mathbf{I}\mathbb{N}_{\infty}^{\boldsymbol{g}[j]}$ be a proper dense subsemigroup of a Hausdorff topological inverse semigroup $S$. Then $I=S\setminus \mathbf{I}\mathbb{N}_{\infty}^{\boldsymbol{g}[j]}$  is a topological group.
\end{theorem}

\begin{proof}
We claim that the ideal $I$ contains a unique idempotent.

Suppose to the contrary that $I$ has at least two distinct idempotent $e$ and $f$. By Lemma~\ref{lemma-2.5} without loss of generality we may assume that the set $V(e)\cap E(\mathscr{C}_\mathbb{N})$ is infinite for any open neighbourhood $V(e)$ of $e$ in $S$. Since $S$ is an inverse semigroup $ef=fe=h$ for some $h\in I\cap E(S)$. Fix an arbitrary open neighbourhood $U(h)$ in $S$. Then there exist disjoint open neighbourhoods $W(e)$ and $W(f)$ of the points $e$ and $f$ in $S$, respectively, such that $W(e)\cdot W(f)\subseteq U(h)$.
Since $S$ is Hausdorff, we can additionally assume that $W(e)\cap U(h)=\varnothing$ if $e\neq h$ and $W(f)\cap U(h)=\varnothing$ if $f\neq h$. Since $e\neq f$ we conclude that $W(e)\cap U(h)=\varnothing$ or $W(f)\cap U(h)=\varnothing$.
Since the set $W(f)\cap E\big(\mathbf{I}\mathbb{N}_{\infty}^{\boldsymbol{g}[j]}\big)$ is infinite and for any positive integer $n_0$ there exist finitely many idempotents $\iota\in \mathbf{I}\mathbb{N}_{\infty}^{\boldsymbol{g}[j]}$ such that $n_{\iota}^{\mathbf{d}}=n_0$, we conclude that the set
$
  \big\{\underline{n}_{\iota}^{\mathbf{d}}\colon \iota\in W(f)\cap E\big(\mathbf{I}\mathbb{N}_{\infty}^{\boldsymbol{g}[j]}\big)\big\}
$
is infinite as well. Also, the choice of the neighbourhood $W(e)$ implies that the set
$
  \big\{\underline{n}_{\iota}^{\mathbf{d}}={n}_{\iota}^{\mathbf{d}}\colon \iota\in W(e)\cap E\big(\mathscr{C}_\mathbb{N}\big)\big\}
$
is infinite, too. Then the semigroup operation in $\mathbf{I}\mathbb{N}_{\infty}^{\boldsymbol{g}[j]}$ implies that there exist idempotents $\iota_e\in W(e)$ and $\iota_f\in W(f)$ such that $\iota_e\in \iota_e\cdot W(f)$ and $\iota_f\in \iota_f\cdot W(e)$,
which implies $W(e)\cap U(h)\neq\varnothing\neq W(f)\cap U(h)$. But this contradicts the choice of the neighbourhoods $W(e)$, $W(f)$, $U(h)$.

Since $S$ is an inverse semigroup, we have that $xx^{-1}=x^{-1}x=e$ for any $x\in I$. This implies that $I$ is a subgroup of $S$ with the unit element $e$. Also, the continuity of semigroup operation and the inversion in $S$ implies that $I$ is a topological group with the induced topology from $S$.
\end{proof}

Lemma \ref{lemma-2.7} follows from the definition of an element $\overrightarrow{\gamma}$ for an arbitrary $\gamma\in \mathscr{I}_{\infty}^{\,\Rsh\!\!\!\nearrow}(\mathbb{N})$.

\begin{lemma}\label{lemma-2.7}
For any $\gamma\in \mathscr{I}_{\infty}^{\,\Rsh\!\!\!\nearrow}(\mathbb{N})$ the following statements hold:
\begin{itemize}
  \item[$(i)$] $\overrightarrow{\gamma}\in \mathscr{C}_\mathbb{N}$;
  \item[$(ii)$] ${\overrightarrow{\gamma}}^{-1}=\overrightarrow{\gamma^{-1}}$;
  \item[$(iii)$] $\gamma{\overrightarrow{\gamma}}^{-1}=\overrightarrow{\gamma}{\overrightarrow{\gamma}}^{-1}$;
  \item[$(iv)$] ${\overrightarrow{\gamma}}^{-1}\gamma={\overrightarrow{\gamma}}^{-1}\overrightarrow{\gamma}$.
\end{itemize}
\end{lemma}

\begin{proposition}\label{proposition-2.8}
Let $j$ be any positive integer and $\mathbf{I}\mathbb{N}_{\infty}^{\boldsymbol{g}[j]}$ be a proper dense subsemigroup of a Hausdorff topological inverse semigroup $S$. Then the unique idempotent of $S\setminus\mathbf{I}\mathbb{N}_{\infty}^{\boldsymbol{g}[j]}$ commutes with all elements of the semigroup $\mathbf{I}\mathbb{N}_{\infty}^{\boldsymbol{g}[j]}$.
\end{proposition}

\begin{proof}
By Theorem~\ref{theorem-2.6}, $S\setminus\mathbf{I}\mathbb{N}_{\infty}^{\boldsymbol{g}[j]}$ is a group. Put $e_0$ be the unique idempotent of $S\setminus\mathbf{I}\mathbb{N}_{\infty}^{\boldsymbol{g}[j]}$. Also, by Lemma~\ref{lemma-2.5} the set $U(e_0)\cap E(\mathscr{C}_\mathbb{N})$ is infinite for any open neighbourhood $U(e_0)$ of the point $e_0$ in $S$. This implies that $e_0\in \operatorname{cl}_{S}(\mathscr{C}_\mathbb{N})$. Then by Proposition~III.2 of \cite{Eberhart-Selden-1969}, $e_0\cdot \gamma=\gamma\cdot e_0$ for any $\gamma\in \mathscr{C}_\mathbb{N}$.

Fix an arbitrary $\gamma\in \mathbf{I}\mathbb{N}_{\infty}^{\boldsymbol{g}[j]}$. By Lemma~\ref{lemma-2.7} we have that
\begin{equation*}
\overrightarrow{\gamma}\cdot{\overrightarrow{\gamma}}^{-1}\cdot\gamma=\gamma\cdot{\overrightarrow{\gamma}}^{-1}\cdot\overrightarrow{\gamma}= \overrightarrow{\gamma}\in \mathscr{C}_\mathbb{N}.
\end{equation*}
Since $S$ is an inverse semigroup and $S\setminus\mathbf{I}\mathbb{N}_{\infty}^{\boldsymbol{g}[j]}$ is an ideal of $S$, Lemma~\ref{lemma-2.7} implies that
\begin{align*}
  e_0\cdot \gamma&= \big(e_0\cdot \overrightarrow{\gamma}\cdot{\overrightarrow{\gamma}}^{-1}\big)\cdot \gamma=\\
   &= e_0\cdot \big(\overrightarrow{\gamma}\cdot{\overrightarrow{\gamma}}^{-1}\cdot \gamma\big)=\\
   &= e_0\cdot \overrightarrow{\gamma}=\\
   &= \overrightarrow{\gamma}\cdot e_0 =\\
   &= \big(\gamma\cdot{\overrightarrow{\gamma}}^{-1}\cdot\overrightarrow{\gamma}\big)\cdot e_0 =%\\
   \end{align*}
   \begin{align*}
   &= \gamma\cdot\big({\overrightarrow{\gamma}}^{-1}\cdot\overrightarrow{\gamma}\cdot e_0\big) =\\
   &= \gamma\cdot e_0.
\end{align*}
This completes the proof of the proposition.
\end{proof}

\begin{corollary}\label{corollary-2.9}
Let $j$ be any positive integer and $\mathbf{I}\mathbb{N}_{\infty}^{\boldsymbol{g}[j]}$ be a proper dense subsemigroup of a Hausdorff topological inverse semigroup $S$. Then the group $S\setminus\mathbf{I}\mathbb{N}_{\infty}^{\boldsymbol{g}[j]}$ contains a dense cyclic subgroup.
\end{corollary}

\begin{proof}
By Proposition~\ref{proposition-2.8}, the unique idempotent $e_0$ of $S\setminus\mathbf{I}\mathbb{N}_{\infty}^{\boldsymbol{g}[j]}$ commutes with all elements of the semigroup $\mathbf{I}\mathbb{N}_{\infty}^{\boldsymbol{g}[j]}$ and hence the map $\mathfrak{h}\colon S\to S\setminus\mathbf{I}\mathbb{N}_{\infty}^{\boldsymbol{g}[j]}$, $(\gamma)\mathfrak{h}=e_0\cdot\gamma$ is a homomorphisms. Since $S\setminus\mathbf{I}\mathbb{N}_{\infty}^{\boldsymbol{g}[j]}$ is a subgroup of $S$, by Corollary~1.32 of \cite{Clifford-Preston-1961-1967} the image $(\mathbf{I}\mathbb{N}_{\infty}^{\boldsymbol{g}[j]})\mathfrak{h}$ is a cyclic group. Also, since $\mathbf{I}\mathbb{N}_{\infty}^{\boldsymbol{g}[j]}$ is a dense subset of a topological semigroup $S$, Proposition~1.4.1 of \cite{Engelking-1989} implies that the image $(\mathbf{I}\mathbb{N}_{\infty}^{\boldsymbol{g}[j]})\mathfrak{h}$ is a dense subset of $S\setminus\mathbf{I}\mathbb{N}_{\infty}^{\boldsymbol{g}[j]}$.
\end{proof}
%%%%%%%%%%%%%%%%%%%%%%%%%%%%%%%%%%%%%%%%%%%%%%%%%%%%%%%%%%%%%%%%%%%%%%%%%%%%%%

\section{On a closure of the monoid $\mathbf{I}\mathbb{N}_{\infty}^{\boldsymbol{g}[j]}$ in a locally compact topological inverse semigroup}\label{sec-4}

In \cite{Eberhart-Selden-1969} Eberhart and Selden described  the closure of the bicyclic monoid in a locally compact topological inverse semigroup. We give this description in the terms of the monoid $\mathscr{C}_{\mathbb{N}}$.

\begin{example}\label{example-4.1}
The definition of the bicyclic monoid, its algebraic properties (see \cite[Section~1.12]{Clifford-Preston-1961-1967}) and Remark~\ref{remark-1.1} imply that the following relation
\begin{equation*}
  \gamma\sim\delta \quad \hbox{if and only if}\quad n_{\gamma}^{\mathbf{r}}-n_{\gamma}^{\mathbf{d}}=n_{\delta}^{\mathbf{r}}-n_{\delta}^{\mathbf{d}}, \qquad \gamma,\delta\in \mathscr{C}_{\mathbb{N}},
\end{equation*}
coincides with the minimum group congruence $\mathfrak{C}_{\mathbf{mg}}$ on $\mathscr{C}_{\mathbb{N}}$. Moreover,  the quotient semigroup $\mathscr{C}_{\mathbb{N}}/\mathfrak{C}_{\mathbf{mg}}$ is isomorphic to the additive group of integers $\mathbb{Z}(+)$ by the map
\begin{equation*}
\pi_{\mathfrak{C}_{\mathbf{mg}}}\colon \mathscr{C}_{\mathbb{N}}\to \mathbb{Z}(+), \quad \gamma\mapsto n_{\delta}^{\mathbf{r}}-n_{\delta}^{\mathbf{d}}.
\end{equation*}
The minimum group congruence $\mathfrak{C}_{\mathbf{mg}}$ on $\mathscr{C}_{\mathbb{N}}$ defines the natural partial order $\preccurlyeq$ on the monoid $\mathscr{C}_{\mathbb{N}}$ in the following way:
\begin{equation*}
  \gamma\preccurlyeq\delta \quad \hbox{if and only if}\quad n_{\gamma}^{\mathbf{r}}-n_{\gamma}^{\mathbf{d}}=n_{\delta}^{\mathbf{r}}-n_{\delta}^{\mathbf{d}} \quad \hbox{and} \quad n_{\gamma}^{\mathbf{d}}\geqslant n_{\delta}^{\mathbf{d}}, \qquad \gamma,\delta\in \mathscr{C}_{\mathbb{N}}.
\end{equation*}
We put $\mathcal{C}\mathscr{C}_{\mathbb{N}}=\mathscr{C}_{\mathbb{N}}\sqcup\mathbb{Z}(+)$ and extend the multiplications from the semigroup $\mathscr{C}_{\mathbb{N}}$ and the group $\mathbb{Z}(+)$ onto $\mathcal{C}\mathscr{C}_{\mathbb{N}}$ in the following way:
\begin{equation*}
  k\cdot\gamma=\gamma\cdot k=k+(\gamma)\pi_{\mathfrak{C}_{\mathbf{mg}}}\in \mathbb{Z}(+), \qquad \hbox{for all} \quad k\in\mathbb{Z}(+) \quad \hbox{and} \quad \gamma\in \mathscr{C}_{\mathbb{N}}.
\end{equation*}
Then so defined binary operation is a semigroup operation on $\mathcal{C}\mathscr{C}_{\mathbb{N}}$ such that $\mathbb{Z}(+)$ is an ideal in $\mathcal{C}\mathscr{C}_{\mathbb{N}}$. Also,  this semigroup operation extends the  natural partial order $\preccurlyeq$ from $\mathscr{C}_{\mathbb{N}}$ onto $\mathcal{C}\mathscr{C}_{\mathbb{N}}$ in the following way:
\begin{itemize}
  \item[$(i)$] all distinct elements of $\mathbb{Z}(+)$ are pair-wise incomparable;
  \item[$(ii)$] $k\preccurlyeq\gamma$ if and only if $n_{\gamma}^{\mathbf{r}}-n_{\gamma}^{\mathbf{d}}=k$ for $k\in \mathbb{Z}(+)$ and $\gamma\in \mathscr{C}_{\mathbb{N}}$.
\end{itemize}
For any $x\in \mathcal{C}\mathscr{C}_{\mathbb{N}}$ we denote ${\uparrow_{\preccurlyeq}}x=\left\{y\in \mathcal{C}\mathscr{C}_{\mathbb{N}}\colon x\preccurlyeq y\right\}$.

We define the topology $\tau_{\textsf{lc}}$ on $\mathcal{C}\mathscr{C}_{\mathbb{N}}$ in the following way:
\begin{itemize}
  \item[$(i)$] all  elements of the monoid $\mathscr{C}_{\mathbb{N}}$ are isolated points in $(\mathcal{C}\mathscr{C}_{\mathbb{N}},\tau_{\textsf{lc}})$;
  \item[$(ii)$] for any $k\in \mathbb{Z}(+)$ the family $\mathscr{B}_{\textsf{lc}}(k)=\left\{U_i(k)\colon i\in\mathbb{N}\right\}$, where
  \begin{equation*}
    U_i(k)=\left\{k\right\}\cup\big\{\gamma\in \mathscr{C}_{\mathbb{N}}\colon k\preccurlyeq\gamma \hbox{~and~} n_{\gamma}^{\mathbf{d}}\geqslant i\big\},
  \end{equation*}
  is the base of the topology $\tau_{\textsf{lc}}$ at the point $k\in \mathbb{Z}(+)$.
\end{itemize}

In \cite{Eberhart-Selden-1969} Eberhart and Selden proved that $\tau_{\textsf{lc}}$ is the unique Hausdorff locally compact semigroup inverse topology on $\mathcal{C}\mathscr{C}_{\mathbb{N}}$. Moreover, they shown that if  $\mathscr{C}_{\mathbb{N}}$ is a proper dense subsemigroup of a Hausdorff locally compact topological inverse semigroup $S$, then $S$ is topologically isomorphic to $(\mathcal{C}\mathscr{C}_{\mathbb{N}},\tau_{\textsf{lc}})$.
\end{example}

\begin{example}\label{example-4.2}
Let $\mathcal{C}\mathbf{I}\mathbb{N}_{\infty}^{\boldsymbol{g}[j]}$ be a semigroup defined in Example~\ref{example-1.4}. Put $M$ be an arbitrary subset of $\{2,\ldots,j\}$.

We define the topology $\tau_{\textsf{lc}}^M$ on $\mathcal{C}\mathbf{I}\mathbb{N}_{\infty}^{\boldsymbol{g}[j]}$ in the following way:
\begin{itemize}
  \item[$(i)$] all elements of the monoid $\mathbf{I}\mathbb{N}_{\infty}^{\boldsymbol{g}[j]}$ are isolated points in $\big(\mathcal{C}\mathbf{I}\mathbb{N}_{\infty}^{\boldsymbol{g}[j]},\tau_{\textsf{lc}}^M\big)$;
  \item[$(ii)$] for any $k\in \mathbb{Z}(+)$ the family $\mathscr{B}_{\textsf{lc}}^M(k)=\left\{U_i^M(k)\colon i\in\mathbb{N}\right\}$, where
  \begin{equation*}
    U_i^M(k)=\left\{k\right\}\cup\big\{\gamma\in \mathcal{C}\mathbf{I}\mathbb{N}_{\infty}^{\boldsymbol{g}[j]}[M]\colon k\preccurlyeq\gamma \hbox{~and~} n_{\gamma}^{\mathbf{d}}\geqslant i\big\},
  \end{equation*}
  is the base of the topology $\tau_{\textsf{lc}}^M$ at the point $k\in \mathbb{Z}(+)$.
\end{itemize}
\end{example}

\begin{remark}\label{remark-4.3}
\begin{itemize}
  \item[1.] We observe that a simple verifications show that the following conditions hold:
  \begin{itemize}
    \item[$(i)$] if $k=0$ then
    $%\begin{equation*}
    U_i^M(k)=U_i^M(0)=\left\{0\right\}\cup\big\{\gamma\in \mathcal{C}\mathbf{I}\mathbb{N}_{\infty}^{\boldsymbol{g}[j]}[M]\colon k\preccurlyeq\gamma \hbox{~and~} \gamma\notin {\uparrow_{\preccurlyeq}} \beta^{i-2}\alpha^{i-2}\big\}
    $;%\end{equation*}
    \item[$(ii)$] if $k>0$ then
    $%\begin{equation*}
    U_i^M(k)=\left\{0\right\}\cup\big\{\gamma\in \mathcal{C}\mathbf{I}\mathbb{N}_{\infty}^{\boldsymbol{g}[j]}[M]\colon k\preccurlyeq\gamma \hbox{~and~} \gamma\notin {\uparrow_{\preccurlyeq}} \beta^{i-2}\alpha^{i-2+k}\big\}
    $;%\end{equation*}
    \item[$(iii)$]  if $k<0$ then
    $%\begin{equation*}
    U_i^M(k)=\left\{0\right\}\cup\big\{\gamma\in \mathcal{C}\mathbf{I}\mathbb{N}_{\infty}^{\boldsymbol{g}[j]}[M]\colon k\preccurlyeq\gamma \hbox{~and~} \gamma\notin {\uparrow_{\preccurlyeq}} \beta^{i-2-k}\alpha^{i-2}\big\}
    $.%\end{equation*}
  \end{itemize}

  \item[2.] Since all elements of the monoid $\mathbf{I}\mathbb{N}_{\infty}^{\boldsymbol{g}[j]}$ are isolated points in $\big(\mathcal{C}\mathbf{I}\mathbb{N}_{\infty}^{\boldsymbol{g}[j]},\tau_{\textsf{lc}}^M\big)$ and all distinct elements of the subgroup $\mathbb{Z}(+)$ are incomporable with the respect to the natural partial order on $\mathcal{C}\mathbf{I}\mathbb{N}_{\infty}^{\boldsymbol{g}[j]}$, Proposition~\ref{proposition-1.2} implies that $\tau_{\textsf{lc}}^M$ is a Hausdorff topology on $\mathcal{C}\mathbf{I}\mathbb{N}_{\infty}^{\boldsymbol{g}[j]}$. Also, since for any $\gamma\in \mathscr{C}_{\mathbb{N}}$ the set ${\uparrow_{\preccurlyeq}}\gamma$ is finite we get that $U_i^M(k)$ is compact for any $k\in Z(+)$ and any positive integer $i$. This implies that the space $\big(\mathcal{C}\mathbf{I}\mathbb{N}_{\infty}^{\boldsymbol{g}[j]},\tau_{\textsf{lc}}^M\big)$ is locally comapct, and hence by Theorems~3.3.1, 4.2.9 and Corollary~3.3.6 from \cite{Engelking-1989} it is metrizable.
\end{itemize}
\end{remark}

\begin{proposition}\label{proposition-4.4}
$\big(\mathcal{C}\mathbf{I}\mathbb{N}_{\infty}^{\boldsymbol{g}[j]},\tau_{\textsf{lc}}^M\big)$ is a topological inverse semigroup.
\end{proposition}

\begin{proof}
Since all elements of the monoid $\mathbf{I}\mathbb{N}_{\infty}^{\boldsymbol{g}[j]}$ are isolated points in $\big(\mathcal{C}\mathbf{I}\mathbb{N}_{\infty}^{\boldsymbol{g}[j]},\tau_{\textsf{lc}}^M\big)$ and all distinct elements of the subgroup $\mathbb{Z}(+)$ commute with elements of $\mathbf{I}\mathbb{N}_{\infty}^{\boldsymbol{g}[j]}$, it is suffices to check the continuity of the semigroup operation at the pairs $(\gamma,k_1)$ and $(k_1,k_2)$ where $\gamma\in \mathbf{I}\mathbb{N}_{\infty}^{\boldsymbol{g}[j]}$ and $k_1,k_2\in\mathbb{Z}(+)$.

Fix any $\gamma\in \mathbf{I}\mathbb{N}_{\infty}^{\boldsymbol{g}[j]}$ and $k\in\mathbb{Z}(+)$. Then $\overrightarrow{\gamma}=\beta^p\alpha^r$ for some fixed non-negative integers $p$ and $r$. Hence
\begin{equation*}
\gamma\cdot k=(\gamma)\pi_{\mathfrak{C}_{\mathbf{mg}}}+k= \big(\overrightarrow{\gamma}\big)\pi_{\mathfrak{C}_{\mathbf{mg}}}+k= r-p+k,
\end{equation*}
and for any positive integer $i>\max\left\{p,r\right\}+j$ we have that
\begin{equation*}
 \gamma\cdot U_i^M(k)\subseteq U_i^M(r-p+k).
\end{equation*}

Fix any $k_1,k_2\in\mathbb{Z}(+)$. Then for any positive integer $i>j$ by Proposition~1.4.7 of \cite{Lawson-1998} and Proposition~\ref{proposition-1.7} we have that
$
 U_i^M(k_1)\cdot U_i^M(k_2)\subseteq U_i^M(k_1+k_2).
$

The above arguments and the equality
$
\left(U_i^M(k)\right)^{-1}=U_i^M(-k)
$
complete the proof of the proposition.
\end{proof}

\begin{lemma}\label{lemma-4.5}
Let $j$ be any positive integer and $\mathbf{I}\mathbb{N}_{\infty}^{\boldsymbol{g}[j]}$ be a proper dense subsemigroup of a Hausdorff locally compact topological inverse semigroup $S$. Then $G=S\setminus\mathbf{I}\mathbb{N}_{\infty}^{\boldsymbol{g}[j]}$ is topologically isomorphic to the discrete additive group of integers $\mathbb{Z}(+)$.
\end{lemma}

\begin{proof}
By Corollary~\ref{corollary-2.9}, $G$ is a subgroup of $\mathbf{I}\mathbb{N}_{\infty}^{\boldsymbol{g}[j]}$ which contains a dense cyclic subgroup. By Theorem~\ref{theorem-2.2}, $\mathbf{I}\mathbb{N}_{\infty}^{\boldsymbol{g}[j]}$ is a discrete subspace of $S$, and hence by Theorem~3.3.9 of \cite{Engelking-1989}, $G$ is a closed subspace of $S$. Then Theorem~3.3.8 of \cite{Engelking-1989} and Theorem~\ref{theorem-2.6} imply that $G$ with the induced topology from $S$ is a locally compact topological group. By the Weil Theorem (see \cite{Weil-1938}) the topological group $G$ is either compact or discrete. By Lemma~\ref{lemma-2.5} the remainder $\operatorname{cl}_S(\mathscr{C}_\mathbb{N})\setminus\mathscr{C}_\mathbb{N}$ of the subsemigroup $\mathscr{C}_\mathbb{N}$ in $S$ is non-empty. Then by Theorem~3.3.8 of \cite{Engelking-1989}, $\operatorname{cl}_S(\mathscr{C}_\mathbb{N})$ is a locally compact space. Theorem~V.7 of \cite{Eberhart-Selden-1969} implies that $H=\operatorname{cl}_S(\mathscr{C}_\mathbb{N})\setminus\mathscr{C}_\mathbb{N}$ is a group, which is topologically isomorphic to the discrete additive group of integers $\mathbb{Z}(+)$. By Proposition~1.4.19 of \cite{Arhangelskii-Tkachenko-2008}, $H$ is a closed discrete subgroup of $G$, and hence by Theorem~1.4.23 of \cite{Arhangelskii-Tkachenko-2008} the topological group $G$ is topologically isomorphic to the discrete additive group of integers $\mathbb{Z}(+)$.
\end{proof}

A partial order $\le$ on a topological space $X$ is called \emph{closed} (or \emph{continuous}) if the relation $\le$ is a closed subset of
$X\times X$ in the product topology \cite{Gierz-Hofmann-Keimel-Lawson-Mislove-Scott-2003}. A topological space with a closed partial order  is called a \emph{pospace}.

Later we assume that $\mathbf{I}\mathbb{N}_{\infty}^{\boldsymbol{g}[j]}$ is a proper dense subsemigroup of a Hausdorff locally compact topological inverse semigroup $S$ and we identify the topological group $G$ with the discrete additive group of integers $\mathbb{Z}(+)$.

We observe that equality ${\uparrow_{\preccurlyeq}}k=\left\{ \gamma\in S\colon \gamma\cdot 0=k\right\}$ implies that ${\uparrow_{\preccurlyeq}}k$ is an open-and-closed subset of $S$ for any $k\in \mathbb{Z}(+)$. Since $\mathbf{I}\mathbb{N}_{\infty}^{\boldsymbol{g}[j]}$ is a discrete subspace of $S$ the above arguments and Lemma~\ref{lemma-4.5} imply the following lemma:

\begin{lemma}\label{lemma-4.6}
The  natural partial order $\preccurlyeq$ on $S$ is closed, and moreover ${\uparrow_{\preccurlyeq}}x$ is open-and-closed subset of $S$ for any $x\in S$.
\end{lemma}

\begin{lemma}\label{lemma-4.7}
For any $k,l\in \mathbb{Z}(+)$ the subspace ${\uparrow_{\preccurlyeq}}k$ and ${\uparrow_{\preccurlyeq}}l$ of $S$ are homeomorphic. Moreover, the map $P_{\alpha^k}\colon {\uparrow_{\preccurlyeq}}0\to {\uparrow_{\preccurlyeq}}k$, $x\mapsto x\cdot \alpha^k$ is a homeomorphism for $k>0$, and the map $\Lambda_{\beta^k}\colon {\uparrow_{\preccurlyeq}}0\to {\uparrow_{\preccurlyeq}}k$, $x\mapsto \beta^k\cdot x$ is a homeomorphism for $k<0$.
\end{lemma}

\begin{proof}
Proposition~1.4.7 from \cite{Lawson-1998} implies that the maps $P_{\alpha^k}$ and $\Lambda_{\beta^k}$ are well defined. It is obvious that complete to prove that the second part of the lemma holds. We shall show that the map $P_{\alpha^k}$ determines a homeomorphism from ${\uparrow_{\preccurlyeq}}0$ onto  ${\uparrow_{\preccurlyeq}}k$. In the case of the map $\Lambda_{\beta^k}$ the proof is similar.

We define a map $P_{\beta^k}\colon {\uparrow_{\preccurlyeq}}k\to {\uparrow_{\preccurlyeq}}0$ by the formula $(x)P_{\beta^k}=x\cdot \beta^k$. Then we have that $(0)P_{\alpha^k}=k$ and $(k)P_{\beta^k}=0$. Moreover, we have that $(x)P_{\alpha^k}P_{\beta^k}=x$ for any $x\in {\uparrow_{\preccurlyeq}}0$ and $(y)P_{\beta^k}P_{\alpha^k}=y$ for any $y\in {\uparrow_{\preccurlyeq}}k$. Therefore the compositions of maps $P_{\alpha^k}P_{\beta^k}\colon {\uparrow_{\preccurlyeq}}0\to {\uparrow_{\preccurlyeq}}0$ and $P_{\beta^k}P_{\alpha^k}\colon {\uparrow_{\preccurlyeq}}k\to {\uparrow_{\preccurlyeq}}k$ are identity maps of the sets ${\uparrow_{\preccurlyeq}}0$ and  ${\uparrow_{\preccurlyeq}}k$, respectively. Hence the maps $P_{\alpha^k}$ and $P_{\beta^k}$ are bijections, and hence $P_{\beta^k}$ is inverse of $P_{\alpha^k}$. Since right translations in the topological semigroup $S$ are continuous, the maps $P_{\alpha^k}\colon {\uparrow_{\preccurlyeq}}0\to {\uparrow_{\preccurlyeq}}k$ and $P_{\beta^k}\colon {\uparrow_{\preccurlyeq}}k\to {\uparrow_{\preccurlyeq}}0$ are homeomorphisms.
\end{proof}

By Lemma~\ref{lemma-2.5} the remainder $\operatorname{cl}_S(\mathscr{C}_\mathbb{N})\setminus\mathscr{C}_\mathbb{N}$ of the subsemigroup $\mathscr{C}_\mathbb{N}$ in $S$ is non-empty. Also, Theorem~V.7 of \cite{Eberhart-Selden-1969} implies that the remainder $\operatorname{cl}_S(\mathscr{C}_\mathbb{N})\setminus\mathscr{C}_\mathbb{N}$ is a group, which is topologically isomorphic to the discrete additive group of integers $\mathbb{Z}(+)$. This and results of \cite[Section~V]{Eberhart-Selden-1969} (see Example~\ref{example-4.1}) imply the following proposition:

\begin{proposition}\label{proposition-4.8}
Let $j$ be any positive integer and $\mathbf{I}\mathbb{N}_{\infty}^{\boldsymbol{g}[j]}$ be a proper dense subsemigroup of a Hausdorff locally compact topological inverse semigroup $\big(\mathcal{C}\mathbf{I}\mathbb{N}_{\infty}^{\boldsymbol{g}[j]},\tau\big)$. Then $\tau$ induces the topology $\tau_{\textsf{lc}}$ on the semigroup $\mathcal{C}\mathscr{C}_{\mathbb{N}}$.
\end{proposition}

If $M=\varnothing$ then we denote the locally compact semigroup inverse topology $\tau_{\textsf{lc}}^M$ on the monoid $\mathcal{C}\mathbf{I}\mathbb{N}_{\infty}^{\boldsymbol{g}[j]}$ by $\tau_{\textsf{lc}}^\varnothing$. Also in the case when $M=\{2,\ldots,j\}$ we denote the  topology $\tau_{\textsf{lc}}^M$ on $\mathcal{C}\mathbf{I}\mathbb{N}_{\infty}^{\boldsymbol{g}[j]}$ by $\tau_{\textsf{lc}}^{[2:j]}$.

Proposition~\ref{proposition-4.8} implies the following:

\begin{proposition}\label{proposition-4.9}
Let $j$ be any positive integer and $\mathbf{I}\mathbb{N}_{\infty}^{\boldsymbol{g}[j]}$ be a proper dense subsemigroup of a Hausdorff locally compact topological inverse semigroup $\big(\mathcal{C}\mathbf{I}\mathbb{N}_{\infty}^{\boldsymbol{g}[j]},\tau\big)$. Then $\tau_{\textsf{lc}}^\varnothing\subseteq \tau\subseteq \tau_{\textsf{lc}}^{[2:j]}$.
\end{proposition}

\begin{theorem}\label{theorem-4.10}
Let $j$ be any positive integer and $\mathbf{I}\mathbb{N}_{\infty}^{\boldsymbol{g}[j]}$ be a proper dense subsemigroup of a Hausdorff locally compact topological inverse semigroup $(S,\tau)$. Then $(S,\tau)$ topologically isomorphic to the topological inverse semigroup $\big(\mathcal{C}\mathbf{I}\mathbb{N}_{\infty}^{\boldsymbol{g}[j]},\tau_{\textsf{lc}}^M\big)$ for some subset $M$ of $\{2,\ldots,j\}$.
\end{theorem}

\begin{proof}
Lemma~\ref{lemma-4.5} implies that the inverse semigroup $S$ is isomorphic to the monoid $\mathcal{C}\mathbf{I}\mathbb{N}_{\infty}^{\boldsymbol{g}[j]}$. Also, by the definition of the monoid $\mathbf{I}\mathbb{N}_{\infty}^{\boldsymbol{g}[j]}$, Lemma~\ref{lemma-4.7} and Proposition~\ref{proposition-4.9} we get that there exists a maximal subset $M_1$ of $\{2,\ldots,j\}$ such that the following condition holds:
\begin{itemize}
  \item[$(*)$] for every open neighbourhood $V_0$ of the point $0\in \mathbb{Z}(+)$ in $\big(\mathcal{C}\mathbf{I}\mathbb{N}_{\infty}^{\boldsymbol{g}[j]},\tau\big)$ there exists an open neighbourhood $U_i^{M_1}(0)$ of $0$ in $\big(\mathcal{C}\mathbf{I}\mathbb{N}_{\infty}^{\boldsymbol{g}[j]},\tau_{\textsf{lc}}^{M_1}\big)$ such that $U_i^{M_1}(0)\subseteq V_0$ and $V_0\setminus U_i^{M_1}(0)$ is infinite.
\end{itemize}
Since the topology $\tau$ is locally compact and $\mathbf{I}\mathbb{N}_{\infty}^{\boldsymbol{g}[j]}$ is a discrete subsemigroup of $\big(\mathcal{C}\mathbf{I}\mathbb{N}_{\infty}^{\boldsymbol{g}[j]},\tau\big)$, without loss of generality we may assume that the open neighbourhood $V_0$ is compact.

The maximality of $M_1$ and condition  $(*)$ imply that there exists a subset $M_1^1\subseteq \{2,\ldots,j\}$ such that $M_1\subset M_1^1$, $\left|M_1^1\setminus M_1\right|=1$ and for every open neighbourhood $V_0$ of the point $0\in \mathbb{Z}(+)$ in $\big(\mathcal{C}\mathbf{I}\mathbb{N}_{\infty}^{\boldsymbol{g}[j]},\tau\big)$ the following conditions hold:
\begin{equation}\label{eq-4.1}
\left|\left(V_0\cap U_i^{M_1^1}(0)\right)\setminus U_i^{M_1}(0)\right|=\infty \quad \hbox{and} \quad \left|U_i^{M_1^1}(0)\setminus \left(V_0\cap U_i^{M_1^1}(0)\right)\right|=\infty.
\end{equation}
By continuity of the semigroup operation in $\big(\mathcal{C}\mathbf{I}\mathbb{N}_{\infty}^{\boldsymbol{g}[j]},\tau\big)$ there exists a compact-and-open neighbourhood $U_0\subseteq V_0$ of the point $0\in \mathbb{Z}(+)$ in the space $\big(\mathcal{C}\mathbf{I}\mathbb{N}_{\infty}^{\boldsymbol{g}[j]},\tau\big)$ such that $\beta\cdot U_0\cdot\alpha\subseteq V_0$. Then the semigroup operation of $\mathcal{C}\mathbf{I}\mathbb{N}_{\infty}^{\boldsymbol{g}[j]}$, the above inclusion and conditions \eqref{eq-4.1} imply that the set $V_0\setminus U_0$ is infinite, which contradicts the compactness of $V_0$. This and maximality of $M_1$ imply that the set $V_0\setminus U_i^{M_1}(0)$ is finite for every open neighbourhood $V_0$ of the point $0\in \mathbb{Z}(+)$ in $\big(\mathcal{C}\mathbf{I}\mathbb{N}_{\infty}^{\boldsymbol{g}[j]},\tau\big)$ and any open neighbourhood $U_i^{M_1}(0)$ of $0$ in $\big(\mathcal{C}\mathbf{I}\mathbb{N}_{\infty}^{\boldsymbol{g}[j]},\tau_{\textsf{lc}}^{M_1}\big)$. Then the bases of $\tau$ and $\tau_{\textsf{lc}}^{M_1}$ at the point $0\in \mathbb{Z}(+)$ coincide, and hence by Lemma~\ref{lemma-4.7} we get that $\tau=\tau_{\textsf{lc}}^{M_1}$.
\end{proof}

\begin{corollary}\label{corollary-4.11}
For any positive integer $j$ there exists exactly
$2^{j-1}$ pairwise topologically non-isomorphic Hausdorff locally compact semigroup inverse
topologies on the monoid $\mathcal{C}\mathbf{I}\mathbb{N}_{\infty}^{\boldsymbol{g}[j]}$.
\end{corollary}
%%%%%%%%%%%%%%%%%%%%%%%%%%%%%%%%%%%%%%%%%%%%%%%%%%%%%
%\medskip
\section*{Acknowledgements}
The authors acknowledge  Alex Ravsky and the referee for useful important comments and suggestions.
%%%%%%%%%%%%%%%%%%%%%%%%%%%%%%%%%%%%%%%%%%%%%%%%%%%%%%%%%%%%


\begin{thebibliography}{XX}
\bibitem{Anderson-Hunter-Koch-1965}
L.~W.~Anderson, R.~P.~Hunter, and R.~J.~Koch,
\emph{Some results on stability in semigroups}.
Trans. Amer. Math. Soc. {\bf 117} (1965), 521--529.

\bibitem{Arhangelskii-Tkachenko-2008}
A. Arhangel'skii and M. Tkachenko,
\emph{Topological Groups and Related Structures},
Atlantis, 2008.

\bibitem{Banakh-Dimitrova-Gutik-2009}
T.~Banakh, S.~Dimitrova, and O.~Gutik,
\emph{The Rees-Suschkiewitsch Theorem for simple
topological semigroups}, Mat. Stud. \textbf{31}  (2009), no. 2, 211--218.

\bibitem{Banakh-Dimitrova-Gutik-2010}
T.~Banakh, S.~Dimitrova, and O.~Gutik,
\emph{Embedding the bicyclic semigroup into countably compact topological semigroups},
Topology Appl. \textbf{157} (2010), no.~18, 2803--2814.

\bibitem{Bardyla-Ravsky-2019}
S. Bardyla and A. Ravsky,
\emph{Closed subsets of compact-like topological spaces},
Appl. Gen. Topol. \textbf{21} (2020), no. 2, 201--214.
%doi:10.4995/agt.2020.12258.


\bibitem{Bertman-West-1976}
M.~O.~Bertman and T.~T.~West,
{\it Conditionally compact bicyclic semitopological semigroups},
Proc. Roy. Irish Acad. {\bf A76} (1976), no. 21--23, 219--226.


\bibitem{Bezushchak-2004}
O. Bezushchak,
\emph{On growth of the inverse semigroup of partially defined co-finite automorphisms of integers},
Algebra Discrete Math. (2004), no.~2, 45--55.


\bibitem{Bezushchak-2008}
O. Bezushchak,
\emph{Green's relations of the inverse semigroup of partially defined cofinite isometries of discrete line},
Visn., Ser. Fiz.-Mat. Nauky, Kyiv. Univ. Im. Tarasa Shevchenka (2008), no.~1, 12--16.


\bibitem{Carruth-Hildebrant-Koch-1983-1986}
J.~H.~Carruth, J.~A.~Hildebrant, and  R.~J.~Koch,
\emph{The Theory of Topological Semigroups}, Vol. I, Marcel
Dekker, Inc., New York and Basel, 1983; Vol. II, Marcel Dekker,
Inc., New York and Basel, 1986.

\bibitem{Brown-1965}
D. R. Brown, \emph{Topological semilattices on the two-cell},
Pacific J. Math. \textbf{15} (1965), no. 1, 35--46.

\bibitem{Chuchman-Gutik-2010}
I. Ya. Chuchman and O. V. Gutik,
\emph{Topological monoids of almost monotone injective co-finite partial selfmaps of the set of positive integers}.
Carpathian Math. Publ. \textbf{2} (2010), no.~1, 119--132.


\bibitem{Clifford-Preston-1961-1967}
A. H.~Clifford and  G. B.~Preston,
\emph{The Algebraic Theory of Semigroups},
Vol. I., Amer. Math. Soc. Surveys 7, Pro\-vidence, R.I., 1961; Vol. II., Amer. Math. Soc. Surveys 7, Providence, R.I., 1967.


%\bibitem{Doroshenko-2005}
%V. Doroshenko,
%\emph{Generators and relations for the semigroups of increasing functions on $\mathbb{N}$ and $\mathbb{Z}$}.
%Algebra Discrete Math. (2005), no.~4,  1--15.

%\bibitem{Doroshenko-2009}
%V. V. Doroshenko,
%\emph{On semigroups of order-preserving transformations of countable linearly ordered sets}.
%Ukr. Math. J. \textbf{61} (2009), no.~6, 859--872.
%doi:10.1007/s11253-009-0256-3
%(translation of Ukr. Mat. Zh. \textbf{61} (2009), no.~6, 723--732 (in Ukrainian)).

\bibitem{Eberhart-Selden-1969}
C. Eberhart and J. Selden,
\emph{ On the closure of the bicyclic semigroup},
Trans. Amer. Math. Soc. {\bf 144} (1969), 115--126.

\bibitem{Engelking-1989}
R.~Engelking,
\emph{General Topology}, 2nd ed.,
Heldermann, Berlin, 1989.

\bibitem{Gierz-Hofmann-Keimel-Lawson-Mislove-Scott-2003}
G. Gierz, K. H. Hofmann, K. Keimel, J. D. Lawson, M. Mislove, and D. S. Scott,
\emph{Continuous Lattices and Domains},
Cambridge Univ. Press, 2003.

\bibitem{Gutik-2015}
O. Gutik,
\emph{On the dichotomy of a locally compact semitopological bicyclic monoid with adjoined zero},
Visnyk L'viv Univ., Ser. Mech.-Math. \textbf{80} (2015), 33--41.



\bibitem{Gutik-Repovs-2007}
O.~Gutik and D.~Repov\v{s},
\emph{On countably compact $0$-simple topological inverse semigroups},
Semigroup Forum \textbf{75} (2007), no.~2, 464--469.


\bibitem{Gutik-Repovs-2011}
O.~Gutik and D.~Repov\v{s},
\emph{Topological monoids of monotone, injective partial selfmaps of $\mathbb{N}$ having cofinite domain and image},
Stud. Sci. Math. Hungar. \textbf{48} (2011), no.~3, 342--353.
%doi: 10.1556/SScMath.48.2011.3.1176

\bibitem{Gutik-Repovs-2012}
O.~Gutik and D.~Repov\v{s},
\emph{On monoids of injective partial selfmaps of integers with cofinite domains and images},
Georgian Math. J. \textbf{19} (2012),  no.~3, 511--532.
%doi:10.1515/gmj-2012-0022

\bibitem{Gutik-Repovs-2015}
O.~Gutik and D.~Repov\v{s},
\emph{On monoids of injective partial cofinite selfmaps},
Math. Slovaca \textbf{65} (2015), no.~5,  981--992.
%doi: 10.1515/ms-2015-0067


\bibitem{Gutik-Savchuk-2017}
O.~Gutik and A.~Savchuk,
\emph{On the semigroup $\mathbf{ID}_{\infty}$,}
Visn. Lviv. Univ., Ser. Mekh.-Mat. \textbf{83} (2017), 5--19 (in Ukrainian).

\bibitem{Gutik-Savchuk-2018}
O.~Gutik and A.~Savchuk,
\emph{The semigroup of partial co-finite isometries of positive integers,}
 Bukovyn. Mat. Zh. \textbf{6} (2018), no. 1--2,  42--51 (in Ukrainian).
%doi: 10.31861/bmj2018.01.042

\bibitem{Gutik-Savchuk-2019}
O. Gutik and A. Savchuk,
\emph{On inverse submonoids of the monoid of almost monotone injective co-finite partial selfmaps of positive integers},
Carpathian Math. Publ. \textbf{11} (2019), no. 2, 296--310. %doi: 10.15330/cmp.11.2.296-310

\bibitem{Gutik-Savchuk-2020}
O. Gutik and A. Savchuk,
\emph{On the monoid of cofinite partial isometries of $\mathbb{N}$ with the usual metric},
Visn. Lviv. Univ., Ser. Mekh.-Mat. \textbf{89} (2020), 17--30.

\bibitem{Hildebrant-Koch-1986}
J.~A.~Hildebrant and R.~J.~Koch,
{\it Swelling actions of $\Gamma$-compact semigroups}, Semigroup
Forum {\bf 33} (1986), 65--85.

\bibitem{Koch-Wallace-1957}
R.~J.~Koch and A.~D.~Wallace,
\emph{Stability in semigroups},
Duke Math. J. \textbf{24} (1957), no. 2, 193--195.


\bibitem{Lawson-1998}
M.~Lawson,
\emph{Inverse Semigroups. The Theory of Partial Symmetries},
Singapore: World Scientific, 1998.


\bibitem{Petrich-1984}
M.~Petrich,
\emph{Inverse Semigroups},
John Wiley $\&$ Sons, New York, 1984.

\bibitem{Ruppert-1984}
W.~Ruppert,
\emph{Compact Semitopological Semigroups: An Intrinsic Theory},
Lect. Notes Math., \textbf{1079}, Springer, Berlin, 1984.


\bibitem{Wagner-1952}
V.~V. Wagner,
\textit{Generalized groups},
Dokl. Akad. Nauk SSSR \textbf{84} (1952), 1119--1122 (in Russian).

\bibitem{Weil-1938}
A. Weil,
\emph{L'int\'{e}gration dans les groupes topologiques et ses applications},
Actualites Scientifiques No. 869,  Hermann, Paris, 1940.

\end{thebibliography}
\end{document}